\documentclass[11pt]{amsart}
\usepackage[parfill]{parskip}    % Activate to begin paragraphs with an empty line rather than an indent
\usepackage{graphicx}
\usepackage{amssymb}
\usepackage{epstopdf}
\usepackage{enumitem,amsmath,amsfonts,latexsym,hyperref}
\usepackage{stmaryrd}

\newtheorem{theorem}{Theorem}
\newtheorem{corollary}[theorem]{Corollary}
\newtheorem{lemma}[theorem]{Lemma}
\newtheorem{proposition}[theorem]{Proposition}
\newtheorem*{lem:kirkgoebel}{Lemma \ref{lem:kirkgoebel}}

\theoremstyle{definition}
\newtheorem{definition}[theorem]{Definition}
\newtheorem{remark}[theorem]{Remark}

\newtheorem{example}[theorem]{Example}

\makeatletter
\newcommand{\dotminus}{\mathbin{\text{\@dotminus}}}

\newcommand{\@dotminus}{%
  \ooalign{\hidewidth\raise1ex\hbox{.}\hidewidth\cr$\m@th-$\cr}%
}
\makeatother

\numberwithin{equation}{section}
\numberwithin{theorem}{section}
\title[An application of continuous logic to fixed point theory]{An application of continuous logic \\ to fixed point theory}
\author{Simon Cho}
\begin{document}
\maketitle

\begin{abstract}
In aiming to apply to a broader class of examples the Avigad-Iovino ``ultraproducts and metastability'' approach to obtaining uniformity for convergence of sequences, we construct a framework using continuous logic that in particular is able to handle discontinuous functions in its domain of discourse. This setup weakens the usual continuity requirements for functions, but compensates for the loss of control by introducing a notion of ``linear structure'' that captures in a quite general way the situation of having geodesics between every pair of points, and has as a special case the vector space structure of Banach spaces. We use this to apply the Avigad-Iovino method to specific convergence results from functional analysis involving iterations of discontinuous functions, and so obtain uniform metastable convergence in those results.
\end{abstract}

\section{Introduction}
%\subsection{}

Kohlenbach and others (\cite{meta1}, \cite{meta2}, \cite{general}, \cite{KBRS}, \cite{kohlenbach1}, \cite{kohlenbach2}, \cite{kohlenbach3}) have applied ``proof mining'' techniques to various convergence and fixed point existence results found in e.g. functional analysis to extract computable and uniform bounds from proofs that do not a priori provide such information. Here ``uniform'' is taken to mean ``uniform in the specific functions/operators and the spaces on which they act'', modulo obvious size features such as moduli of continuity or diameters of the spaces.

Motivated by these earlier approaches, Avigad (one of the authors of \cite{meta1}, \cite{meta2}) and Iovino used the model-theoretic machinery of continuous first order logic, in which a metric on the space replaces the equality predicate, to show at least the existence of such uniformity in many of the cases to which Kohlenbach's proof mining technique applies \cite{AI}.

On the one hand, the Avigad-Iovino approach is more conveniently accessible to mathematicians working in fields other than logic. On the other, the continuous logic framework powering this elegant approach imposes rather stringent uniform continuity requirements on its objects of discourse. Indeed, Kohlenbach notes (for example in \cite{kohlenbachprogress}) two advantages of his own method: one, that his proof mining is able to provide, in fact compute, the actual uniform bound, and two, the proof mining method is in a sense more robust in that it can treat cases in which the function or operator in question may have some desirable properties but is possibly discontinuous.

It is this second point that the current paper addresses: we develop a setup (which we term the \emph{geodesic framework}) using continuous logic that relaxes the usual continuity requirements in a precise sense but introduces a formalized notion of ``linear structure'' which allows us to sufficiently compensate for the resulting loss of control in the absence of continuity. Using this setup we are able to apply the Avigad-Iovino method to a broader class of examples, in particular cases (\cite{GLS}, \cite{dne}, \cite{suzuki}) in which the function in question is allowed to be discontinuous. These examples, successfully treated via the proof-theoretic approach in \cite{kohlenbachfejer}, were previously out of reach of the Avigad-Iovino model-theoretic approach.

In order to provide context for the specific applications mentioned above, we first consider the following illustration of the Avigad-Iovino method:

\begin{example}
\label{example: intro}
Let $B$ be a reflexive Banach space and consider an operator $T: B \rightarrow B$. For $f \in B$, we have its $n^{\text{th}}$ ergodic average $A_n f  = \frac{1}{n} \sum\limits_{m < n} T^m f$.

A version of the mean ergodic theorem states that if $T$ is power bounded (i.e. $\exists M$ such that $||T^n|| \leq M$ for all $n \in \mathbb{N}$), then given any element $f$ of $B$, the sequence $\{A_n f\}$ of ergodic averages converges.

(That is, there is some $K: \mathbb{R}_{>0} \rightarrow \mathbb{N}$ such that given any $\epsilon > 0$, for all $i,j \geq K(\epsilon)$, we have that $||A_i f - A_j f|| < \epsilon$.)
\end{example}

One might ask if there is some sense in which the above convergence is uniform across all such spaces $B$, operators $T$, and elements $f$ of $B$. If we are asking for uniformity in the sense of Cauchy convergence, i.e. for a $K$ (in the notation of Example \ref{example: intro}) that is independent of the specific choice of $B$, $T$, and $f$, the answer is a resounding no: it is known that this convergence can be made arbitrarily slow, even with bounds on obvious size features like the diameter of $B$ or the norm of $f$ (and $T$ is already power bounded) \cite{nouniform}. However, we might ask for a weaker uniformity, in the following sense:

\begin{definition}
\label{def: metastability}
%(\cite{meta1}, \cite{meta2}, \cite{AI}, \cite{tao})

Let $\{x_n\}$ be a sequence of points in a metric space $(X,d)$.

Given a function $F: \mathbb{N} \rightarrow \mathbb{N}$, we say that $b_F: \mathbb{R}_{>0} \rightarrow \mathbb{N}$ is a \emph{bound on the rate of metastability} of $\{x_n\}$ with respect to $F$ if for each $\epsilon >0$ there exists an $n \leq b_F(\epsilon)$ such that for all $i,j \in [n, F(n)]$, we have that $d(x_i, x_j) < \epsilon$.

If such a bound $b_F$ exists, we say that $\{x_n\}$ \emph{converges metastably} with respect to $F$.
\end{definition}

\begin{remark}

The first explicit bounds on metastability were extracted in \cite{klambov}, after which many other papers in proof theory on this topic were published, among them that of Avigad, Gerhardy, and Towsner in \cite{meta2}; the name ``metastability'' is due to Tao \cite{tao}. From the point of view of logic, metastability is a special case of Kreisel's no-counterexample interpretation \cite{kreisel1}, \cite{kreisel2}.

\end{remark}

It is easy to verify that a sequence $\{x_n\}$ converges in the Cauchy sense if and only if it converges metastably with respect to every $F: \mathbb{N} \rightarrow \mathbb{N}$:

\begin{proposition}
\label{prop: cauchy=meta}

Let $\{x_n\}$ be a sequence of points in a metric space $(X,d)$. The following are equivalent:

\begin{enumerate}[label=(\alph*), ref=(\alph*)]

	\item There exists some $K: \mathbb{R}_{>0} \rightarrow \mathbb{N}$ such that for every $\epsilon > 0$ and for all $i, j \geq K(\epsilon)$, we have that $d(x_i, x_j) < \epsilon$. \label{cauchyconv}

	\item For each $F: \mathbb{N} \rightarrow \mathbb{N}$, $\{x_n\}$ converges metastably with respect to $F$. \label{metaconv}

\end{enumerate}

\end{proposition}

\begin{proof}
\ref{cauchyconv} $\Rightarrow$ \ref{metaconv}: Given any $F: \mathbb{N} \rightarrow \mathbb{N}$, define $b_F$ by $b_F(\epsilon) = K(\epsilon)$.

\ref{metaconv} $\Rightarrow$ \ref{cauchyconv}: Assume that $\{x_n\}$ fails to be Cauchy convergent, i.e. there is some $\epsilon > 0$ such that for every $n \in \mathbb{N}$, we can find $i_n, j_n \geq n$ such that $d(x_{i_n}, x_{j_n}) \geq \epsilon$. Let us define $F: \mathbb{N} \rightarrow \mathbb{N}$ as $F(n) = \max (i_n, j_n)$. Then $\{x_n\}$ fails to be metastably convergent for this $F$.
\end{proof}

Therefore if a convergence result (e.g. the mean ergodic theorem) guarantees convergence for a class $\mathcal{C}$ of pairs $((X,d), \{x_n\})$ satisfying certain conditions, then - despite not having uniform Cauchy convergence in the sense of having a $K$ (in the notation of Example \ref{example: intro} and Proposition \ref{prop: cauchy=meta} \ref{cauchyconv}) that is uniform across all members of $\mathcal{C}$ - we might ask whether, once we specify some $F: \mathbb{N} \rightarrow \mathbb{N}$, there is a bound $b_F$ on the rate of metastability with respect to this $F$ which is uniform across $\mathcal{C}$.

In the case of the mean ergodic theorem, if we restrict to certain reasonable classes $\mathcal{C}$ of Banach spaces $B$ (e.g. the class of uniformly convex Banach spaces for a fixed modulus of uniform convexity) and ergodic averages of points in a uniformly bounded subset of each $B$, the question above has a positive answer, as shown in \cite{AI} using continuous logic:

\begin{theorem}
\label{thm: ergodic}
\textup{(\cite{AI})}

Let $\mathcal{C}$ be any class of Banach spaces with the property that the ultraproduct of any countable collection of elements of $\mathcal{C}$ is a reflexive Banach space. For every $\rho > 0$, $M$, and function $F: \mathbb{N} \rightarrow \mathbb{N}$, there is a bound $b$ such that the following holds: given any Banach space $B$ in $\mathcal{C}$, any linear operator $T$ on $B$ satisfying $||T^n|| \leq M$ for every $n$, any $f \in B$, and any $\epsilon > 0$, if $||f|| / \epsilon \leq \rho$, then there is an $n \leq b$ such that $||A_i f - A_j f|| < \epsilon$ for every $i, j \in [n, F(n)]$, where $A_k = \frac{1}{k} \sum\limits_{m < k} T^m f$.

\end{theorem}

Notice, in particular, that the operator $T: B \rightarrow B$ above is uniformly continuous. This allows for the problem to be formalized in continuous logic, and the additional conditions on $B$, $T$, and $f \in B$ then guarantee that the above particular sequence of iterations involving $T$ converges to a fixed point. One then finds via an argument that crucially utilizes the continuous ultraproduct (see Theorem \ref{thm: avigadiovino}) that there is a \emph{uniform} bound on the metastability of this convergence that is independent of the particular choice of $B$, $T$, and $f$.

In \cite{GLS}, \cite{dne}, \cite{suzuki} one has a similar situation except that $T$ is in general discontinuous, and so prevents the problem from being formalized in continuous logic, which requires all functions to come with moduli of uniform continuity. Specifically, consider the following:

\begin{example}
\label{example: intro2}
\cite{GLS}, \cite{suzuki}
\begin{enumerate}[label=(\alph*), ref=(\alph*)]
	\item Let $B$ be a Banach space, $C \subset B$ a bounded convex subset, and $T: C \rightarrow C$ a function satisfying, for some fixed $\lambda \in (0,1)$,
	\begin{equation*}
	\forall x,y \in C, \quad \lambda ||x - Tx|| \leq ||x-y|| \Rightarrow ||Tx-Ty|| \leq ||x-y||.
	\end{equation*}
	Then \cite{suzuki} shows that given any $x_1 \in C$, the sequence $\{x_n\}$ given by $x_{n+1} = (1-\lambda) x_n + \lambda Tx_n$ satisfies $d(x_n, Tx_n) \rightarrow 0$. \label{introafps}

	\item If in addition to the above we also have that $C$ is compact and, for some fixed $\mu \geq 1$,
	\begin{equation*}
	\forall x,y \in C, \quad d(x,Ty) \leq \mu \, d(x, Tx) + d(x,y),
	\end{equation*}
	then \cite{GLS} shows that the sequence $\{x_n\}$ of \ref{introafps} converges to a fixed point $x$ of $T$. \label{introcondemu}
\end{enumerate}
\end{example}

One might ask for a uniform bound on the rate of metastability for the convergence $d(x_n, Tx_n) \rightarrow 0$ of \ref{introafps} and for the convergence $\{x_n\} \rightarrow x$ of \ref{introcondemu}. However, this problem is not formulable in continuous logic, because $T$ is in general discontinuous; Example \ref{example: toy} gives a simple instance of $T: C \rightarrow C$ satisfying both \ref{introafps} (for $\lambda = \frac{1}{2}$) and \ref{introcondemu} (for $\mu = 3$) \cite{suzuki}.

The idea behind our approach (the ``geodesic framework'') to resolving this issue is first to notice that the analytic arguments in the proofs of Example \ref{example: intro2} all revolve around the construction and properties of the sequence $\{x_n\}$ of ``iterated linear interpolations'', which only relies on the underlying vector space structure. We abstract this vector space structure to a general ``linear structure'' defined on (pseudo)metric spaces that interacts with the (pseudo)metric as expected, while dropping the continuity requirement for functions (but not for predicates and connectives). In doing so, we are able to (1) formalize classes of examples such as the above which depend not on continuity of functions but rather on an underlying linear structure on the space, while (2) preserving all of the necessary properties of the usual continuous logic that enable the Avigad-Iovino method to apply to such examples. In particular, we obtain (general versions of) the following uniformization of Example \ref{example: intro2}:

\begin{theorem}
\label{thm:intromain}
\quad
\begin{enumerate}[label=(\alph*), ref=(\alph*)]
	\item Let $B$ be a Banach space, $C \subset B$ a convex subset with diameter bounded above by some fixed $D$, and $T: C \rightarrow C$ a function satisfying the condition of Example \ref{example: intro2} \ref{introafps}.

Given any $x_1 \in C$, let $\{x_n\}$ be the sequence defined by $x_{n+1} = (1-\lambda) x_n + \lambda Tx_n$.

Then given $F: \mathbb{N} \rightarrow \mathbb{N}$, there is a bound $b_F$ on the rate of metastability for the sequence $d_n = d(x_n, Tx_n)$, which is uniform across all choices of $B$, $C$, $T$, and $x_1$ satisfying the above conditions.

	\item If in addition to the above we have that $C$ is totally bounded with some fixed modulus of total boundedness $\beta: \mathbb{N} \rightarrow \mathbb{N}$ and $T$ satisfies the condition of Example \ref{example: intro2} \ref{introcondemu}, then given $F: \mathbb{N} \rightarrow \mathbb{N}$ there is a bound $b_F$ on the rate of metastability for the sequence $\{x_n\}$ which is uniform across all choices of $B$, $C$, $T$, and $x_1$ satisfying the above conditions.
\end{enumerate}

\end{theorem}

We describe the geodesic framework in Sections \ref{section: setup} and \ref{section:newultraproduct} after outlining the features of the usual continuous logic and the Avigad-Iovino method in Section \ref{preliminaries}. We then describe in detail the analytic aspects of the examples of \cite{GLS}, \cite{dne}, \cite{suzuki} in Section \ref{section: analysis}. Finally in Section \ref{section:synthesis} we show that the geodesic framework is indeed able to handle the relevant features of such examples, thus enabling the Avigad-Iovino approach to yield (Theorem \ref{thm:main1} and Theorem \ref{thm:main2}) a uniform bound on the rate of metastability for the sequences in question.

We thus illustrate the applicability of the geodesic framework with specific examples involving iterated linear interpolation, and given the prevalence of linear interpolation arguments in e.g. functional analysis, we expect this setup to meaningfully broaden the scope of the applicability of metric model theory to such disciplines.

The author is deeply grateful to his graduate advisor Henry Towsner for his helpful insights and invaluable guidance throughout the process of writing this paper.

\section{Preliminaries}
\label{preliminaries}

For completeness' sake, we first describe the features of continuous first order logic relevant to our current interests: more details can be found in e.g. \cite{BBHU}. Those who are already familiar with continuous logic should feel free to skip ahead to the next section.

\begin{definition}
\label{def:signature}
(\cite{BBHU})

A \emph{continuous signature} $S$ consists of the following data:

\begin{enumerate}[label=(\alph*), ref=(\alph*)]

	\item A \emph{metric symbol} $d$ and a nonnegative real number (the \emph{upper bound on diameter}) $D$.
	
	\item An index set $I$ and for each $i \in I$, a positive integer $n_i$, an \emph{($n_i$-ary) function symbol} $f_i$, and a modulus of uniform continuity $\delta_{f_i}: \mathbb{R}_{> 0} \rightarrow \mathbb{R}_{>0}$.
	
	\item An index set $J$ and for each $j \in J$, a positive integer $n_j$, a nonnegative real number $a_j$, an \emph{($n_j$-ary) predicate symbol} $R_j$, and a modulus of uniform continuity $\delta_{R_j}: \mathbb{R}_{>0} \rightarrow \mathbb{R}_{>0}$.
	
	\item An index set $K$ and for each $k \in K$ a \emph{constant symbol} $c_k$.

\end{enumerate}

When clear from context, we may simply refer to $S$ as a \emph{signature}.

%Conversely, if $S$ is some signature and $\mathcal{X}$ is some metric structure for whom the symbols of $S$ satisfy the above conditions, then $\mathcal{X}$ is called an \emph{$S$-structure}.

\end{definition}

\begin{definition}
\label{def:contlogic}
(\cite{BBHU})

Let $S$ be a (continuous) signature, with notation for its constituent data as in Definition \ref{def:signature}.

An \emph{$S$-structure} $\mathcal{X}$ consists of the following data:

	\begin{enumerate}[label=(\alph*), ref=(\alph*)]
	
	\item A complete, bounded metric space $(X, d_X)$ where $d_X$ is the metric on $X$, and an assignment of the metric symbol $d$ of $S$ to the metric $d_X$; moreover the diameter of $X$ must be bounded above by $D$.
	
	\item For each $i \in I$ an assignment of the function symbol $f_i$ to a uniformly continuous function $\llbracket f_i \rrbracket: X^{n_i} \rightarrow X$ which has $\delta_{f_i}$ as a modulus of uniform continuity.
	
	We call $\llbracket f_i \rrbracket$ an \emph{($n_i$-ary) function} on $X$, and by abuse of notation may refer to it as $f_i$ when clear from context.
	
	\item For each $j \in J$ an assignment of the predicate symbol $R_j$ to a uniformly continuous function $\llbracket R_j \rrbracket: X^{n_j} \rightarrow [0, a_j]$ which has $\delta_{R_j}$ as a modulus of uniform continuity.
	
	We call $\llbracket R_j \rrbracket$ an \emph{($n_j$-ary) predicate} on $X$, and by abuse of notation may refer to it as $R_j$ when clear from context.

	\item For each $k \in K$ an assignment of the constant symbol $c_k$ to a point $\llbracket c_k \rrbracket \in X$.
	
	We call $\llbracket c_k \rrbracket$ a \emph{constant} on $X$, and by abuse of notation may refer to it as $c_k$ when clear from context.
	
	\end{enumerate}

We may also call $\mathcal{X}$ a \emph{metric structure}.

\end{definition}

In the above, we always consider $X^n$ as equipped with the maximum metric, i.e. $d_{X^n}(x,y) = \max\limits_{1 \leq m \leq n} d(x_m, y_m)$ for $x = (x_1, \dots, x_n)$ and $y = (y_1, \dots, y_n)$.

Following the authors of \cite{BBHU}, we will assume for simplicity's sake that $D = 1$ and $a_j = 1$ throughout. Also, we reserve the right to abuse notation by reusing the index set $I$ in other contexts possibly unrelated to the above definitions.

Given a signature $S$ - which specifies the vocabulary of the language in which we can speak - we can talk of (first-order) formulae and sentences in the language. First, we say that the \emph{logical symbols} of $S$ include $d$ (which plays the role of equality in classical first-order logic, where $d(x,y) = 0$ is analogous to the classical statement $x = y$); an infinite set $V_S = \{x_i \mid i \in I\}$ of variables, for $I$ some index set (a priori unrelated to the set indexing the function symbols of $S$); a symbol $u$ for each continuous function $u: [0,1]^n \rightarrow [0,1]$ (which together play the role of $n$-ary connectives); and the symbols $\sup$ and $\inf$ which are analogous to the classical quantifies $\forall$ and $\exists$, respectively.

We then say that the \emph{nonlogical symbols} of $S$ are the function, predicate, and constant symbols of $S$. The \emph{cardinality} $|S|$ of $S$ is the smallest infinite number $\geq$ the cardinality of the set of nonlogical symbols of $S$.

\begin{definition}
\label{def:formulae}
(\cite{BBHU})

Let $S$ be a signature.

\begin{enumerate}[label=(\alph*), ref=(\alph*)]

	\item A \emph{term} for $S$ is given by the following inductive description:
	
	\begin{enumerate}[label=(\roman*), ref=(\roman*)]
	
		\item Each variable and each constant symbol is a term.
		
		\item $f(t_1, \dots, t_n)$ is a term when $f$ is some ($n$-ary) function symbol and each $t_i$ is itself a term.
	
	\end{enumerate}
	
	\item An \emph{atomic formula} for $S$ is given by an expression of the form $P(t_1, \dots, t_n)$ where $P$ is some ($n$-ary) predicate symbol and each $t_i$ is a term. (The symbol $d$ for the metric is treated as a binary predicate symbol.)
	
	\item A \emph{formula} for $S$ is given by the following inductive description:
	
	\begin{enumerate}[label=(\roman*), ref=(\roman*)]
	
		\item Each atomic formula is a formula.
		
		\item $u(\phi_1, \dots, \phi_n)$ is a formula when $u$ is some $n$-ary connective, i.e. a continuous function $[0,1]^n \rightarrow [0,1]$, and each $\phi_i$ is a formula.
		
		\item $\sup\limits_x \phi$ and $\inf\limits_x \phi$ are each formulae when $x$ is a variable and $\phi$ is a formula.
	
	\end{enumerate}

\end{enumerate}

\end{definition}
	
Many notions from classical first order logic carry over unmodified; \emph{subformulae} of a formula, \emph{occurrence} of a variable within a (sub)formula, and \emph{substitution} of a term for a variable are a few examples. We then say that if a variable $x$ occurs in a formula $\phi$ and $x$ is not contained in any subformula of the form $\sup\limits_x \phi^\prime$ or $\inf\limits_x \phi^\prime$ (i.e. $x$ is not quantified over), then $x$ is a \emph{free variable} in $\phi$. A formula $\phi$ having no free variables is called a \emph{sentence}.

Often we will write a term $t$ as $t(x_1, \dots, x_n)$ to make it clear which (distinct) variables occur in $t$. Similarly we write a formula $\phi$ as $\phi(x_1, \dots, x_n)$ to make it clear which are the (distinct) free variables occurring in $\phi$. A term with no variables is called a \emph{closed term}.

Given a signature $S$ with its attendant logical and nonlogical symbols, and a correspondence between $S$ and a metric structure $\mathcal{X}$, it is clear what the interpretation of each term and formula should be, since they are built up inductively out of functions, predicates, and constants, the interpretation of which is a priori given via the aforementioned correspondence. For complete details, see \cite{BBHU}.

It is straightforward to verify that, from the moduli of uniform continuity of all the functions and predicates that occur in a given formula, we can find a modulus of uniform continuity for that formula, and likewise for terms.

Given two $S$-formulae $\phi(x_1, \dots, x_n)$ and $\psi(x_1, \dots, x_n)$, we define their \emph{logical distance} $|\phi - \psi|$ as
\begin{equation*}
| \phi - \psi | = \sup\limits_{\mathcal{X}; x_1, \dots, x_n \in \mathcal{X}} | \phi(x_1, \dots, x_n) - \psi (x_1, \dots, x_n) |
\end{equation*}
and $\phi$, $\psi$ are said to be \emph{logically equivalent} when $|\phi - \psi| = 0$.

It is possible to restrict our (a priori uncountable) set of logical connectives to a more manageable, countable set of connectives with a very compact description using the above notion of logical distance and density with respect to said distance, and then to talk about ``definable'' predicates (and functions, subsets, etc.) - but we will not outline this direction in this paper, and instead refer the interested reader to \cite{BBHU} for details.

In continuous logic, we call formulae (resp. sentences) of the form $\phi = 0$ \emph{conditions} (resp. closed conditions). These play the same role that formulae and sentences play in the usual first-order logic. If $\phi$ and $\psi$ are formulae then we can regard formulae of the form $\phi = \psi$ as shorthand for the condition $| \phi - \psi | = 0$. We can thus regard formulae of the form $\phi = r$ for $r \in [0,1]$ as a special case of this, by considering $r$ as a $0$-ary connective. In continuous logic we are usually content with models satisfying ``arbitrarily close'' approximations to a given condition $\phi = r$, so it suffices to restrict the set of $0$-ary connectives $r$ to $\mathbb{Q} \cap [0,1]$.

Similarly, we can regard $\phi \leq \psi$ as the condition $\phi \dotminus \psi = 0$, where $t_1 \dotminus t_2 = \max(t_1 - t_2, 0)$. If $\Sigma$ is a set of conditions, then we denote by $\Sigma^+$ the set of conditions $\phi \leq \frac{1}{n}$ for each $n \in \mathbb{N}$ and each formula $\phi$ such that $\phi = 0$ is in $\Sigma$. If $\Sigma$ is a set of closed conditions, then we say that $\mathcal{X}$ is a \emph{model} of $\Sigma$ when $\mathcal{X}$ satisfies every condition in $\Sigma$, where the notion of ``satisfaction'' of a condition by a structure $\mathcal{X}$ is the obvious analogue of ``satisfaction'' as defined in usual first-order logic. Clearly $\mathcal{X}$ is a model of $\Sigma$ if and only if it is a model of $\Sigma^+$.

We now describe ultraproducts (in the sense of \cite{BBHU}), as they occupy a central role in both \cite{AI} and this paper. For completeness' sake, we start by defining ultrafilters:

\begin{definition}
\label{def: ultrafilter}

Let $I$ be a set.

\begin{enumerate}[label=(\alph*), ref=(\alph*)]

	\item A (proper) \emph{filter} on $I$ is a set $\mathcal{F} \subset \mathcal{P} (I)$ (where $\mathcal{P}$ gives the powerset of its argument, and so $\mathcal{F} \in \mathcal{P}(\mathcal{P}(I))$) satisfying the following:
	\label{filter}
	
	\begin{enumerate}[label=(\roman*), ref=(\roman*)]
	
	\item $\emptyset \notin \mathcal{F}$. \label{proper}
	
	\item $\mathcal{F}$ is upward closed, i.e. if $A \in \mathcal{F}$ and $A \subset B$, then $B \in \mathcal{F}$.
	
	\item $\mathcal{F}$ is closed under finite intersection, i.e. if $A, B \in \mathcal{F}$ then $A \cap B \in \mathcal{F}$.
	
	\end{enumerate}
	
	\item A (proper) filter $\mathcal{F}$ on $I$ is an \emph{ultrafilter} on $I$ if for every $A \in \mathcal{P}(I)$, either $A \in \mathcal{F}$ or $I \setminus A \in \mathcal{F}$. \label{ultra}
	
\end{enumerate} 

\end{definition}

The condition \ref{ultra} of the above definition for a filter $\mathcal{F}$ to be an ultrafilter given above is equivalent to $\mathcal{F}$ being a maximal filter, where $\mathcal{P}(\mathcal{P}(I))$ is partially ordered with respect to inclusion. An ultrafilter $\mathcal{F}$ is called \emph{principal} if it is the ultrafilter generated by a singleton set, i.e. $\mathcal{F} = \{A \in \mathcal{P}(I) \mid i_0 \in A\}$ for some $i_0 \in I$ (and of course, $\mathcal{F}$ is called \emph{nonprincipal} if it is not principal). In all of our constructions involving ultrafilters, we will assume that our ultrafilter is nonprincipal.

Sometimes it is more convenient to talk of a \emph{filter base}, where we say that $\mathcal{F}^\prime \subset \mathcal{P}(I)$ is a base for a filter $\mathcal{F}$ (or that $\mathcal{F}^\prime$ \emph{generates} $\mathcal{F}$) if $\mathcal{F}^\prime$ satisfies \ref{filter}\ref{proper} of the above definition, and it is downward directed, i.e. for $A, B \in \mathcal{F}^\prime$, there is some $C \in \mathcal{F}^\prime$ such that $C \subset A \cap B$. $\mathcal{F}$ is then the minimal filter containing $\mathcal{F}^\prime$, i.e. $\mathcal{F} = \{A \in \mathcal{P}(I) \mid A^\prime \subset A, A^\prime \in \mathcal{F}^\prime\}$. A popular example of a nonprincipal ultrafilter is any ultrafilter containing the cofinite filter on $\mathbb{N}$ (i.e. the filter $\{A \in \mathcal{P}(\mathbb{N}) \mid \mathbb{N} \setminus A \textup{ is finite}\}$).

Before we actually define the ultraproduct construction, we should note a few facts which we will require. Let $X$ be a topological space, and $\{x_i\}$ some family of points in $X$, indexed by a set $I$. Let $\mathcal{F}$ be an ultrafilter on $I$. We say that $x = \lim\limits_{i, \mathcal{F}} x_i$ or that $x$ is a $\mathcal{F}$-limit of the family $\{x_i\}$ when for every neighborhood $U$ of $x$, we have $\{ i \mid x_i \in U\} \in \mathcal{F}$. If $X$ is Hausdorff, this limit must be unique.

\begin{definition}
\label{def: ultraproduct}
(\cite{BBHU})

Let $S$ be some signature, and $\mathcal{X}_i$ a family of $S$-structures, indexed by some set $I$. Let $\mathcal{F}$ be an ultrafilter on $I$.

Let $\tilde{X} = \prod\limits_i X_i$ be the cartesian product of the underlying spaces of $\mathcal{X}_i$. There is an induced function $d: \tilde{X} \times \tilde{X} \rightarrow [0,1]$ given by $d((x_i), (y_i)) = \lim\limits_{i, \mathcal{F}} d_i (x_i, y_i)$.

Let $\sim_\mathcal{F}$ be the equivalence relation on $\tilde{X}$ given by $x \sim y \Leftrightarrow d(x,y) = 0$, and let $X = \tilde{X} / \sim_{\mathcal{F}}$.

We call $X$ the \emph{$\mathcal{F}$-ultraproduct} of the spaces $X_i$. If all the $X_i$ are the same, then we also call $X$ their \emph{$\mathcal{F}$-ultrapower}.

\end{definition}

For each function, predicate, and constant symbol in $S$ in the above definition, we have a family $\{f_i\}$, $\{R_i\}$, $\{c_i\}$ of functions, predicates, and constants interpreting those symbols in each $\mathcal{X}_i$. For each such family of objects, the above construction induces a corresponding ultraproduct object. That is, given a family $\{f_i: X_i \rightarrow X_i\}$ of functions, we have a function $f: X \rightarrow X$ defined as $f(x) = [(f_i(x_i))]_\mathcal{F}$ where $(x_i)$ is a representative of the equivalence class of $x$ in $X$ and $[(f_i(x_i))]_\mathcal{F}$ is the equivalence class of $(f_i(x_i)) \in \prod\limits_i X_i$. (That $f$ is well-defined follows from the fact that all the $f_i$ share the same modulus of uniform continuity and from the way $\sim$ is defined.) Note that this $f$ shares the same modulus of uniform continuity with each of the $f_i$.

Similarly we have that the $\{R_i\}$ define a predicate $R$ in the ultraproduct, and that the $c_i$ define a constant $c$ in the ultraproduct. Thus given a family of structures $\{ \mathcal{X}_i \}$, we have not only an ultraproduct of their underlying spaces, but an ultraproduct of $S$-structures, which is itself an $S$-structure.

Ultraproducts feature prominently in \cite{AI} as well as this paper, in large part due to the following variant of \L{o}\'{s}'s theorem, the moral content of which is that ``a statement is true of the ultraproduct if and only if it is mostly true of its factors.''

\begin{theorem}
\label{thm: los}
\textup{(\cite{BBHU})}

Let $S$ be a signature, and $\{\mathcal{X}_i\}$ an $I$-indexed family of $S$-structures. Let $\mathcal{F}$ be an ultrafilter on $I$, and $\mathcal{X}$ the $\mathcal{F}$-ultraproduct of the $\{\mathcal{X}_i\}$ having $X$ as its underlying space.

Let $\phi(x)$ be an $S$-formula, with $\{a_i\}$ a family of elements of $X_i$. Let $a$ be the corresponding element in $X$. Then:
\begin{equation*}
\phi(a) = \lim\limits_{i, \mathcal{F}} \phi(a_i)
\end{equation*}
\end{theorem}

The proof of the above theorem, which is actually more general (it is true of formulae $\phi$ depending on any number $n$ of free variables), is through induction on the complexity of formulae.

Our (and \cite{AI}'s) interest in Theorem \ref{thm: los} lies in leveraging it to obtain the following theorem (due to \cite{AI} but rephrased slightly here to better reflect the underlying logical machinery), which is the main ingredient of the proof of Theorem \ref{thm: ergodic}:

\begin{theorem}
\label{thm: avigadiovino}
\textup{(\cite{AI})}

Let $S$ be a signature, and let $\{t_n\}$ be a sequence of closed $S$-terms.

Let $\mathcal{C}$ be a collection of $S$-structures $\mathcal{X}$, and for each $\mathcal{X}$ let $\{x_n\}$ denote the interpretation in $\mathcal{X}$ of the sequence $\{t_n\}$.

Finally, let $\mathcal{F}$ be an ultrafilter on $\mathbb{N}$. Then the following are equivalent:

\begin{enumerate}[label=(\alph*), ref=(\alph*)]

\item For every $\epsilon > 0$ and every $F: \mathbb{N} \rightarrow \mathbb{N}$, there is some $b \geq 1$ such that the following holds: for every $\mathcal{X}$ in $\mathcal{C}$, there is an $n \leq b$ such that $d(x_i, x_j) < \epsilon$ for every $i, j \in [n, F(n)]$. \label{uniformtosequence}

\item For any sequence $\{\mathcal{X}_k\}$ of elements of $\mathcal{C}$, let $\mathcal{X}$ be their $\mathcal{F}$-ultraproduct. Then for every $\epsilon > 0$ and every $F: \mathbb{N} \rightarrow \mathbb{N}$,  there is an $n$ such that $d(x_i, x_j) < \epsilon$ for every $i, j \in [n, F(n)]$. \label{sequencetouniform}

\end{enumerate}

\end{theorem}

\begin{proof}
\ref{uniformtosequence} $\Rightarrow$ \ref{sequencetouniform}: For any fixed $\frac{1}{2} \epsilon > 0$ and any fixed $F: \mathbb{N} \rightarrow \mathbb{N}$, there is some $b \geq 1$ such that every member $\mathcal{X}$ of $\mathcal{C}$ satisfies the condition

$\min\limits_{n \leq b} \left ( \max\limits_{i,j \in [n, F(n)]} d(x_i, x_j) \right ) \leq \frac{1}{2}\epsilon$, or more formally,

$\min\limits_{n \leq b} \left ( \max\limits_{i,j \in [n, F(n)]} (d(x_i, x_j) \dotminus \frac{1}{2}\epsilon) \right ) = 0$.

Since every member of $\mathcal{C}$ is a model of the above condition, any ultraproduct of members of $\mathcal{C}$ must again be a model of this condition.

\ref{sequencetouniform} $\Rightarrow$ \ref{uniformtosequence}: If for some $\epsilon > 0$ and some $F: \mathbb{N} \rightarrow \mathbb{N}$ there is no bound $b$ such as in \ref{uniformtosequence}, then for each $k \in \mathbb{N}$, there is some $\mathcal{X}_k \in \mathcal{C}$ that is a counterexample to $k$ being such a bound. That is, for each $\mathcal{X}_k$, there is an $n \leq k$ such that $d_k(x^k_i, x^k_j) \geq \epsilon$ for some $i, j \in [n, F(n)]$. Let $\mathcal{X}$ be the $\mathcal{F}$-ultraproduct of these structures $\mathcal{X}_k$.

Given any $n$, since there are cofinitely many $k \geq n$, there are cofinitely many $k$ such that there exist $i,j \in [n, F(n)]$ with $d_k(x^k_i, x^k_j) \geq \epsilon$. It follows that there is some specific pair $i,j \in [n, F(n)]$ such that $d_k(x^k_i, x^k_j) \geq \epsilon$ for $\mathcal{F}$-many $k$, so that $d(x_i,x_j) = \lim\limits_{k, \mathcal{F}} d_k(x^k_i, x^k_j) \geq \epsilon$ for that choice of $i, j$. Since $n$ was arbitrary, we see that \ref{sequencetouniform} fails.
\end{proof}

The starring role of the continuous ultraproduct in this crucial theorem illustrates why uniform continuity is necessary in applying the Avigad-Iovino approach to obtaining uniformity. Let us consider a toy example due to \cite{suzuki} which shows what can happen in the absence of uniform continuity:

\begin{example}
\label{example: toy}
(\cite{suzuki})

Let $T: [0,3] \rightarrow [0,3]$ be defined by $Tx =
\begin{cases}
0 &\text{ for } x \neq 3\\
1 &\text{ for } x = 3
\end{cases}
$.

Let $\mathcal{F}$ be an ultrafilter containing the cofinite filter on $\mathbb{N}$, and let $([0,3])_\mathcal{F}$ denote the ultrapower of $[0,3]$ with respect to this ultrafilter.

The sequence $\{a_n = 3\}$ represents the same point in the ultrapower as the sequence $\{b_n = 3 - \frac{1}{n}\}$, while the sequences $\{Ta_n = 1\}$ and $\{Tb_n = 0\}$ represent different points. That is, the ultrapower of the function $T$ fails to be well-defined. (This kind of phenomenon is precisely what having a modulus of uniform continuity would prevent.)
\end{example}

Although the function given in Example \ref{example: toy} is discontinuous, it is an instance of a function that is well behaved in other ways:

\begin{definition}
\label{def: condc}
(\cite{suzuki})

Let $X$ be a Banach space and $C$ a nonempty subset. A function $T: C \rightarrow X$ is said to satisfy \emph{condition (C)} when for all $x,y \in C$,

$\frac{1}{2} ||x - Tx|| \leq ||x-y||$ implies $||Tx - Ty|| \leq ||x-y||$.

\end{definition}

Any nonexpansive mapping satisfies condition (C), but condition (C) is clearly weaker. For instance, it is easily verified that the function in Example \ref{example: toy} satisfies condition (C). \cite{GLS} and \cite{suzuki} show how this condition can be leveraged, in the presence of certain other topological conditions, to yield the existence of a fixed point to which a certain kind of iteration sequence converges - we will describe this in detail in Section \ref{section: analysis}. The point is that a function might be discontinuous yet satisfy conditions that guarantee convergence to a fixed point, and so we might ask if the Avigad-Iovino approach to showing that such convergence is uniform (in the sense of Theorem \ref{thm: ergodic}) could be adapted to settings in which the objects in question are allowed to be discontinuous yet are nevertheless ``nice'' in other ways, given the relative convenience of said approach.

We will show in this paper that this is indeed possible, by setting up the problem appropriately within continuous logic. Part of our setup will amount to weakening the equivalence relation we quotient by when taking the ultraproduct. The usual equivalence relation forces the resulting ultraproduct to be a strict metric space (which leads to problems of the type we have seen above), while our setup produces an ultraproduct which is only a pseudometric space in general. In order to preserve the full strength of the usual continuous logic e.g. in applications to fixed point theory of Banach spaces, we must show that the relevant convergence proofs (which were given in the context of complete \emph{normed} vector spaces, which are in particular \textit{metric} spaces with a kind of hyperbolic linear structure) actually apply to the pseudometric case as well, in many cases with very little or even no modification.

\section{The general setup}
\label{section: setup}

\begin{definition}
Let $X$ be a set. A function $d: X \times X \rightarrow \mathbb{R}_{\geq 0}$ is a \emph{pseudometric} for $X$ when it satisfies the following conditions:

\begin{enumerate}[label=(\alph*), ref=(\alph*)]

\item $\forall x \in X, d(x,x) = 0$.

\item $\forall x,y \in X, d(x,y) = d(y,x)$.

\item $\forall x,y,z \in X, d(x,z) \leq d(x,y) + d(y,z)$.

\end{enumerate}

By a \emph{pseudometric space} we refer to a pair $(X, d)$ where $X$ is a set and $d$ is a pseudometric for $X$.

\end{definition}

\begin{remark}

We will frequently have occasion to talk about \emph{bounded} pseudometric spaces, i.e. spaces $(X,d)$ where the pseudometric $d$ takes values in some bounded interval $[0, D]$ for some positive real number $D$. We call $D$ a \emph{bound} for the space $X$, and by abuse of notation we may consider $d: X \times X \rightarrow \mathbb{R}_{\geq 0}$ as instead a function $d: X \times X \rightarrow [0,D]$.

\end{remark}

Note that every pseudometric space is naturally a topological space (the set of $\epsilon$-balls $\{y \mid d(x,y) < \epsilon \}$ for each $x \in X$ and each $\epsilon > 0$ is a basis for the topology on $X$ associated with the pseudometric), and that pseudometric spaces are general enough to include normed vector spaces (in particular, Banach spaces) as a special case.

Just as the usual continuous logic is built upon the theory and properties of metric spaces as its foundation, our setup will have pseudometric spaces as its foundation. We will often find that the pseudometric spaces we are interested in have additional structure (e.g. vector space structure) which features meaningfully in our investigations of them. We note one particular type of such structures, which is a generalization of the vector space structure of a normed vector space.

\begin{definition}
\label{def:linearstructure}
A pseudometric space $(X, d)$ is said to be equipped with a \emph{linear structure} $L$ when there is a specified function $L: X \times X \times [0,1] \rightarrow X$ satisfying the following:

	\begin{enumerate}[label=(\alph*), ref=(\alph*)]
	\item For every $x,y \in X$, the map $L(x,y, \frac{1}{d(x,y)} ( \, \cdot \,)) : [0, d(x,y)] \rightarrow X$ is a geodesic between $x$ and $y$, i.e. it is an isometric embedding satisfying $L(x,y,0) = x$ and $L(x,y, 1) = y$. \label{isometric}
%For each pair $x, y \in X$ of points, there is an isometric embedding $l: [0, d(x,y)] \rightarrow X$ with $L(x,y,0) = l(0) = x$, $L(x,y,1) = l(d(x,y)) = y$, and $L(x,y,t) = l(t \, d(x,y))$ for each $t \in (0,1)$.
%\item For any other isometric embedding $l^\prime: [0, d(x,y)] \rightarrow X$ satisfying $l^\prime(0) = l(0) = x$ and $l^\prime(d(x,y)) = l(d(x,y)) = y$, we have that $d(l(c), l^\prime(c)) = 0$ for all $c \in (0, d(x,y))$.
	\item $d(L(x,y,t), L(y,x,1-t)) = 0$. \label{symmetric}
	\end{enumerate}

\end{definition}

If $(X,d)$ is a pseudometric space with linear structure $L$ we will sometimes refer to it as $(X, d, L)$; and when the context is clear we might say that ``$X$ is a space with linear structure.''

Given these notions, let us now describe the basics of our setup, which we call the \emph{optionally continuous framework}, or OCF for short (note in particular the two different types of functions allowed):

\begin{definition}
\label{def:oclsig}

A \emph{optionally continuous signature} $S$ consists of the following data:

\begin{enumerate}[label=(\alph*), ref=(\alph*)]

	\item A \emph{metric symbol} $d$ and a nonnegative real number (the \emph{upper bound on diameter}) $D$.

	\item An index set $I$ and for each $i$, a positive integer $n_i$ and an \emph{($n_i$-ary) function symbol} $T_i$.
	
	\item An index set $I^\prime$ and for each $i^\prime \in I^\prime$, a positive integer $n_{i^\prime}$, an \emph{($n_{i^\prime}$-ary) continuous function symbol} $f_{i^\prime}$, and a modulus of uniform continuity $\delta_{f_{i^\prime}}: \mathbb{R}_{> 0} \rightarrow \mathbb{R}_{>0}$.
	
	\item An index set $J$ and for each $j \in J$, a positive integer $n_j$, a nonnegative real number $a_j$, an \emph{($n_j$-ary) predicate symbol} $R_j$, and a modulus of uniform continuity $\delta_{R_j}: \mathbb{R}_{>0} \rightarrow \mathbb{R}_{>0}$.
	
	\item An index set $K$ and for each $k \in K$ a \emph{constant symbol} $c_k$.

\end{enumerate}

When clear from context, we may simply refer to $S$ as a \emph{signature}.

%Conversely, if $S$ is some signature and $\mathcal{X}$ is some metric structure for whom the symbols of $S$ satisfy the above conditions, then $\mathcal{X}$ is called an \emph{$S$-structure}.

\end{definition}

\begin{definition}
\label{def:oclobjects}

Let $S$ be an (optionally continuous) signature, with notation for its constituent data as in Definition \ref{def:oclsig}.

An \emph{$S$-structure} $\mathcal{X}$ consists of the following data:

	\begin{enumerate}[label=(\alph*), ref=(\alph*)]
	
	\item A complete, bounded pseudometric space $(X, d_X)$ where $d_X$ is the pseudometric on $X$, and an assignment of the metric symbol $d$ of $S$ to the pseudometric $d_X$; moreover the diameter of $X$ must be bounded above by $D$.

	\item For each $i \in I$ an assignment of the function symbol $T_i$ to a (possibly discontinuous) function $\llbracket T_i \rrbracket: X^{n_i} \rightarrow X$.
	
	We call $\llbracket T_i \rrbracket$ an \emph{($n_i$-ary) function} on $X$, and by abuse of notation may refer to it as $T_i$ when clear from context.
	
	\item For each $i^\prime \in I^\prime$ an assignment of the continuous function symbol $f_{i^{\prime}}$ to a uniformly continuous function $\llbracket f_{i^\prime} \rrbracket: X^{n_{i^\prime}} \rightarrow X$ which has $\delta_{f_{i^\prime}}$ as a modulus of uniform continuity.
	
	We call $\llbracket f_{i^\prime} \rrbracket$ an \emph{($n_{i^\prime}$-ary) continuous function} on $X$, and by abuse of notation may refer to it as $f_{i^\prime}$ when clear from context.
	
	\item For each $j \in J$ an assignment of the predicate symbol $R_j$ to a uniformly continuous function $\llbracket R_j \rrbracket: X^{n_j} \rightarrow [0, a_j]$ which has $\delta_{R_j}$ as a modulus of uniform continuity.
	
	We call $\llbracket R_j \rrbracket$ an \emph{($n_j$-ary) predicate} on $X$, and by abuse of notation may refer to it as $R_j$ when clear from context.

	\item For each $k \in K$ an assignment of the constant symbol $c_k$ to a point $\llbracket c_k \rrbracket \in X$.
	
	We call $\llbracket c_k \rrbracket$ a \emph{constant} on $X$, and by abuse of notation may refer to it as $c_k$ when clear from context.
	
	\end{enumerate}

We may also call $\mathcal{X}$ a \emph{pseudometric structure}.

\end{definition}

Again, as in Definition \ref{def:contlogic}, we consider $X^n$ in Definition \ref{def:oclobjects} above as equipped with the max pseudometric given by $d_{X^n}(x,y) = \max\limits_{1 \leq m \leq n} d(x_m, y_m)$ for $x = (x_1, \dots, x_n)$ and $y = (y_1, \dots, y_n)$. Also as before, we follow the authors of \cite{BBHU} in assuming for simplicity's sake that $D = 1$ and $a_j = 1$ throughout.

With the exception of the ultraproduct, everything else not specifically mentioned above (e.g. terms, formulae, connectives, etc.) remains unchanged from the usual continuous logic. We defer the description of the ultraproduct until Section \ref{section:newultraproduct}.

\begin{remark}
\label{rmk:ref}

The setup presented above may at first glance appear to be a \emph{modification} of continuous logic rather than a special case. However, the optionally continuous framework may be reproduced within the usual continuous logic by utilizing two sorts: one with the correct underlying set of the pseudometric space but equipped with the discrete metric, which surjects onto another sort whose points are the ``distance zero'' equivalence classes of points of the pseudometric space, with the correct distances. The discontinuous parts of the framework take place within the discrete sort, and the rest in the quotient sort. The usual continuous ultraproduct of these structures will coincide with the ``modified'' ultraproduct (described in Section \ref{section:newultraproduct}) of the optionally continuous framework.

The author is grateful to an anonymous referee for pointing out the above fact; however we choose to work with the current (equivalent) presentation of our setup, which (1) we feel is the more natural perspective and (2) makes the technical arguments cleaner.

\end{remark}

As it stands, OCF is not very interesting, since it is essentially a regression back towards classical (non-continuous) first-order logic. However, we can introduce additional structure (namely, linear structure) to OCF to compensate for the control that we lose by allowing for discontinuous functions. Let us call the resulting setup the \emph{geodesic framework}:

\begin{definition}
\label{def:gl}
\item

\begin{enumerate}[label=(\alph*), ref=(\alph*)]

\item A \emph{geodesic signature} $S$ consists of the following data:

	\begin{enumerate}[label=(\roman*), ref=(\roman*)]

	\item An optionally continuous signature $S_0$.

	\item For each $t \in [0,1]$, a \emph{$t$-linear structure symbol} $L_t$.

	\end{enumerate}

\item Given a geodesic signature $S$ with notation for its constituent data as above, an \emph{$S$-structure} $\mathcal{X}$ consists of the following:

	\begin{enumerate}[label=(\roman*), ref=(\roman*)]

	\item An $S_0$-structure $\mathcal{X}_0$.

	\item A linear structure $L$ on the underlying pseudometric space $(X,d)$ of the pseudometric structure $\mathcal{X}_0$, and for each $t \in [0,1]$ an assignment of the $t$-linear structure symbol $L_t$ to the function $\llbracket L_t \rrbracket: X \times X \rightarrow X$ defined by $\llbracket L_t \rrbracket (x,y) = L(x,y,t)$.

	We call each $\llbracket L_t \rrbracket$ the \emph{$t$-value of the linear structure $L$}. By abuse of notation we may refer to it as $L_t$ when clear from context.

	\end{enumerate}

	We may also call $\mathcal{X}$ a \emph{geodesic structure}.	

\end{enumerate}

\end{definition}

As was the case for OCF, with the exception of the ultraproduct (which again, we will describe in Section \ref{section:newultraproduct}), most concepts not specifically mentioned above carry over unchanged from the usual continuous logic. However, for the sake of completeness, we describe terms and formulae in the geodesic framework:

\begin{definition}
\label{def:glformulae}

Let $S$ be a geodesic signature.

\begin{enumerate}[label=(\alph*), ref=(\alph*)]

	\item A \emph{term} for $S$ is given by the following inductive description:
	
	\begin{enumerate}[label=(\roman*), ref=(\roman*)]
	
		\item Each variable and each constant symbol is a term.
		
		\item $T(t_1, \dots, t_n)$ is a term when $T$ is some ($n$-ary) function symbol and each $t_i$ is itself a term.

		\item $f(t_1, \dots, t_n)$ is a term when $f$ is some ($n$-ary) continuous function symbol and each $t_i$ is itself a term.

		\item For each $t \in [0,1]$, $L_t(t_1, t_2)$ is a term when $L_t$ is the $t$-linear structure symbol and $t_1, t_2$ are terms. (That is, $L_t$ is treated as a binary function symbol.)
	
	\end{enumerate}
	
	\item An \emph{atomic formula} for $S$ is given by an expression of the form $P(t_1, \dots, t_n)$ where $P$ is some ($n$-ary) predicate symbol and each $t_i$ is a term. (The symbol $d$ for the pseudometric is treated as a binary predicate symbol.)
	
	\item A \emph{formula} for $S$ is given by the following inductive description:
	
	\begin{enumerate}[label=(\roman*), ref=(\roman*)]
	
		\item Each atomic formula is a formula.
		
		\item $u(\phi_1, \dots, \phi_n)$ is a formula when $u$ is some $n$-ary connective, i.e. a continuous function $[0,1]^n \rightarrow [0,1]$, and each $\phi_i$ is a formula.
		
		\item $\sup\limits_x \phi$ and $\inf\limits_x \phi$ are each formulae when $x$ is a variable and $\phi$ is a formula.
	
	\end{enumerate}

\end{enumerate}

\end{definition}

\begin{remark}
\label{rmk: whytvalues}

The reason for treating the linear structure $L: X \times X \times [0,1] \rightarrow X$ as consisting of separate functions $L_t: X \times X \rightarrow X$ is that, due to the specific technicalities of the ``geodesic ultraproduct'' (which will be addressed in Section \ref{section:newultraproduct}), it is problematic to regard $L$ as simply another function symbol $L$ of OCF. If we were to incorporate $L$ itself as a function symbol, the relationship between \emph{the symbol} $L$ and its interpretation as a linear structure $L$ on $X$ would have to be distinct from that between some function symbol $f$ and its interpretation as a function $f: X \times X \times [0,1] \rightarrow X$, in precisely the manner that has been built in to Definition \ref{def:gl} by considering the linear structure $L$ as a family of functions $L_t$, each of which then receives the same treatment (e.g. under ultraproducts) as the other function symbols do under OCF.

However, when we speak informally of linear structures for geodesic structures and there is no possibility for confusion, we will usually speak of $L$ rather than the family $L_t$ for convenience.

\end{remark}

We note that geodesic structures (with notation as Definition \ref{def:gl}) can be characterized in OCF by the following axioms:

\begin{enumerate}[label=(\alph*), ref=(\alph*)]

	\item For each pair $t, t^\prime \in [0,1]$,

		\begin{enumerate}[label=(\roman*), ref=(\roman*)]

		\item $\sup\limits_x \sup\limits_y \left ( d( L_t (x, y), L_{t^\prime} (x, y) ) \dotminus |t - t^\prime| \, d(x,y) \right ) = 0$, and

		\item $\sup\limits_x \sup\limits_y \left (  |t - t^\prime| \, d(x,y) \dotminus d( L_t (x, y), L_{t^\prime} (x, y) ) \right ) = 0$

		\end{enumerate}

	\item For each $t \in [0,1]$, $\sup\limits_x \sup\limits_y \left ( d( L_t (x, y), L_{1-t}(y,x) ) \right ) = 0$.

\end{enumerate}

It is clear from Definition \ref{def:glformulae} that for $S$ a geodesic signature, if $\phi$ is either an $S$-term or $S$-formula containing only continuous function symbols and predicate symbols (and connectives, which we require to be the same as in continuous logic), then $\phi$ will also have a modulus of uniform continuity.

A given $t$-value $L_t$ of a linear structure $L$ need not be continuous in its arguments. However, there is an interesting particular class of spaces with linear structure satisfying a different niceness condition, as an example of an axiomatizable class in the geodesic framework:

\begin{definition}
\label{def:hyperbolictype}
(\cite{GK}, \cite{takahashi})
A pseudometric space with linear structure $(X,d,L)$ is of \emph{hyperbolic type} when for each quadruple $p,x,y,m \in X$ of points where $m = L(x,y,t)$ for some $t \in [0,1]$, we have that $d(p,m) \leq (1-t) \, d(p,x) + t \, d(p,y)$.
\end{definition}

That spaces of hyperbolic type are axiomatizable in the geodesic framework follows from the easy observation below:

Let $S$ be a geodesic signature, with $t$-linear structure symbols $L_t$. Then (the underlying space of) an $S$-structure $\mathcal{X}$ is of hyperbolic type if and only if $\mathcal{X}$ satisfies, for each $t \in [0,1]$, the $S$-condition
\begin{equation}
\label{hyperbolictypeeqn}
\sup\limits_p \sup\limits_x \sup\limits_y \left ( \left ( d(p, L_t(x,y)) \dotminus (1-t) \, d(p, x) \right ) \dotminus t \, d(p,y) \right ) = 0.
\end{equation}

This condition, as its name suggests, is a notion of hyperbolicity for pseudometric spaces which is general enough to include e.g. CAT(0) spaces as a special case.

One should note that in general, a space may possess many different linear structures, and specifying a geodesic structure on the space simply picks out a favored linear structure. Being of hyperbolic type ensures that this linear structure is nice in the sense of Definition \ref{def:hyperbolictype}, but is a priori a property only of the specified linear structure. In particular, being of hyperbolic type does not preclude the existence of other linear structures; thus it is a weaker condition than many other ``versions'' of hyperbolicity which either imply or explicitly require unique geodesicity (i.e. uniqueness and existence of isometric embeddings of line segments between points).

Indeed, every Banach space (with linear structure given by its vector space structure) is a space of hyperbolic type, while there are many Banach spaces which are not uniquely geodesic and therefore possess multiple linear structures:

\begin{example}
\label{example: nonunique}
Consider $\mathbb{R}^2$ with the supremum (maximum) norm. Between the points $(0,0)$ and $(2,0)$, there is the obvious geodesic $t \mapsto (t, 0)$. However, we can find another geodesic between them given by the piecewise map $t \mapsto (t, t)$ for $t \leq 1$ and $t \mapsto (t, 2-t)$ for $t > 1$.

Thus this Banach space has at least two possible linear structures: one given by the standard vector space structure (call it $L$), and another $L^\prime$ defined by $L^\prime (x, y, t) =\begin{cases}
L(x,y,t) &\text{ for } x,y \notin \{(0,0), (2,0)\}\\
\begin{cases}
(t,t) &\text{ for } t \leq 1\\
(t, 2-t) &\text{ for } t > 1
\end{cases}
& \text{ for } x = (0,0), y = (2,0)
\end{cases}
$
(and $L^\prime(x,y,t) = L^\prime(y,x,1-t)$).
\end{example}

Spaces of hyperbolic type therefore comprise a quite general class of spaces; we give an example of a space with linear structure that fails to be of hyperbolic type:

\begin{example}
\label{example:2-sphere}

Consider $S^2$ with its standard metric. Between any pair of non-antipodal points there is a unique geodesic, and between any pair of antipodal points we can simply pick a geodesic (subject to the symmetry condition of Definition \ref{def:linearstructure} \ref{symmetric}), giving us a linear structure on $S^2$ considered as a pseudometric space.

Fix a point $p \in S^2$ as the ``north pole'', along with a pair of (necessarily non-antipodal) distinct points $x, y \in S^2$ in the open southern hemisphere lying on the same latitude. Let $m$ be the point halfway on the geodesic between $x$ and $y$. Then $d(p, m) > \frac{1}{2} d(p, x) + \frac{1}{2} d(p,y)$ since $d(p,x) = d(p,y)$ and $m$ lies on the great circle between $x$ and $y$.

\end{example}

The phenomenon described in the example above must happen for \emph{any} linear structure on $S^2$ with its standard metric (due to unique geodesicity between non-antipodal points); thus $S^2$ with its standard metric cannot possess any linear structure that makes it a space of hyperbolic type.

The useful property of a space being of hyperbolic type is thus easily translated into the geodesic framework. There are, however, important properties involving the linear structure of a space which are not as readily translated:

\begin{definition}
\label{def:convex}
We say that a subset $C \subset X$ is \emph{convex} (with respect to the linear structure $L$) when for all $x,y \in C$ and for all $t \in [0,1]$, $L(x,y,t) \in C$.
\end{definition}

The property of a subset being convex depends on the specific linear structure: let us again consider the Banach space $\mathbb{R}^2$ of Example \ref{example: nonunique}. Letting $C = [0,2] \times \{0\} \subset \mathbb{R}^2$, clearly $C$ is convex with respect to $L$ but not with respect to $L^\prime$.

This is not the only issue with the notion of convexity. Trying to formalize convexity within the framework of continuous logic (of either the usual or our modified kind) leads immediately to at least the following two questions: (1) how do we formalize the notion of subset, and (2) how do we deal with implication, which is essentially a discontinuous connective, in a logic that only allows uniformly continuous connectives?

The first question has the following answer:

Let $C \subset X$ be a closed subset of a pseudometric space $X$. We can consider a predicate $\hat{C}: X \rightarrow [0,1]$ defined as $\hat{C} (x) = d(x, C) = \inf\limits_{y \in C} d(x, y)$, so that $C = \{ x \mid \hat{C}(x) = 0\}$. It turns out that these kinds of predicates have a nice characterization, the proof of which is irrelevant to our purposes so we refer to interested reader to \cite{BBHU}:

\begin{proposition}
\label{subsetsaspredicates}
\textup{(\cite{BBHU})}

If a predicate $P: X \rightarrow [0,1]$ is of the form $P(x) = d(x, C)$ for some subset $C \subset X$, then it satisfies the following statements which we collectively refer to as \emph{subsets-as-predicates axioms} (or \emph{s.a.p. axioms} for short):

\begin{enumerate}[label=(\alph*), ref=(\alph*)]

\item $\sup\limits_x \inf\limits_y \max (P(y), |P(x) - d(x,y)|) = 0$

\item $\sup\limits_x |P(x) - \inf\limits_y \min (P(y) + d(x,y), 1)| = 0$

\end{enumerate}

Conversely, if a given predicate $P: X \rightarrow [0,1]$ satisfies the s.a.p. axioms, then it is of the form $P(x) = d(x, C)$ where $C = \{x \mid P(x) = 0\}$. Thus there is a one-to-one correspondence between closed subsets of $X$ and predicates on $X$ satisfying the s.a.p. axioms.

\end{proposition}

This correspondence between closed subsets and predicates is what will allow us to (by abuse of notation) speak of them interchangeably without confusion; frequently we will refer to a (closed) subset $C \subset X$ as a predicate $C: X \rightarrow [0,1]$, and vice versa. The advantage of speaking of subsets in terms of predicates is that we can speak of predicates in terms of a given signature without needing to specify a specific structure for that signature. Whenever we have some signature $S$ with a predicate symbol $C$ and some $S$-theory $\Sigma$ containing the s.a.p. axioms for $C$, we call $C$ a \emph{subset predicate} (with respect to $\Sigma$).

Returning to the issue of formalizing convexity, we might now ask ourselves how to deal with the implication in its definition: we need a continuous analogue of the expression
\begin{equation}
\label{eqn: convex}
\forall x, \forall y \, (x \in C \wedge y \in C) \rightarrow L(x,y,t) \in C.
\end{equation}

While Proposition \ref{subsetsaspredicates} gives us a way of finding predicates $C: X \rightarrow [0,1]$ so that we can transform (\ref{eqn: convex}) into
\begin{equation}
\label{eqn: convex2}
\forall x, \forall y \, \left ( \max (C(x), C(y)) = 0 \right ) \rightarrow C(L(x,y,t)) = 0,
\end{equation}

there is no completely satisfactory way to deal with the implication. For example, one might require some (uniformly continuous, monotonically increasing) ``modulus of convexity'' $u: [0,1] \rightarrow [0,1]$ satisfying $u(0) = 0$ and translate (\ref{eqn: convex2}) as

\begin{equation}
\label{eqn: convex3}
\sup\limits_x \sup\limits_y \left ( C (L(x,y,t)) \dotminus u( \max (C(x), C(y)) ) \right ) = 0
\end{equation}

which is certainly a formula expressible in continuous logic. The problem is that (\ref{eqn: convex3}) is an a priori stronger condition than convexity, because even for pairs of points outside of the subset $C$ it requires the potential failure of convexity between those points to be ``no worse'' than their failure to be inside of $C$, in the sense specified by the modulus $u$.

Indeed, from this we see that a faithful translation of convexity would actually be (\ref{eqn: convex3}) with $u$ the \emph{discontinuous} function given by $u(0) = 0$ and $u(x) = 1$ otherwise. But we cannot allow discontinuous connectives (which is what such a $u$ would be), because doing so would mean that ultraproducts (modified or not) of models of some theory $\Sigma$ would no longer necessarily themselves be models of $\Sigma$.

Another solution would be to work with multiple sorts and regard $C$ as its own space with its own linear structure $L_C$ alongside the space $X$ with its linear structure $L_X$, and having an inclusion map $i: C \hookrightarrow X$. We can then say that $C$ is convex when the linear structure of $C$ coincides with that of $X$, i.e. $i \circ L_C = L_X \circ (i \times i \times 1_{[0,1]})$ since by default $C$ must be closed under its own linear structure.

A similar approach is to just regard the subset $C$ as the whole space, and forget about $X$ and questions about the convexity of $C$; this is unproblematic if the behavior/properties of the space $X$ outside of the subset $C$ happen to be unimportant. This is the approach we will take in the applications later in this paper.

We take the last part of this section to address a minor subtlety resulting from our change of setting from metric spaces to pseudometric spaces:

\begin{definition}
\label{def:fixedpt}

Let $(X,d)$ be a pseudometric space and $T: X \rightarrow X$ some map. We say that $p$ is a \emph{fixed point} of $T$ when $d(p, Tp) = 0$.

\end{definition}

That is, ``fixed point'' is understood to mean ``a point that is mapped by $T$ to some (possibly distinct) point at distance $0$'', a necessary weakening of the usual definition of ``fixed point'' since we are working with pseudometrics instead of metrics. In general, this is an ill-behaved notion in the context of pseudometric spaces and arbitrary maps, since then we might have that $d(x, Tx) = 0$ but that possibly $d(x, T^n x) > 0$ for some $n >1$. However, in our applications we will see that the very conditions that guarantee the existence of a fixed point in the above sense also ensure that $d(x, Tx) = 0$ implies $d(x, T^n x) = 0$ for all $n \geq 1$.

This is not a coincidence; fixed point results are commonly obtained through metric arguments that show that the distances between successive terms in a given kind of sequence converge to $0$, and in the presence of the conditions that enable such arguments, we should reasonably expect that $d(x, Tx) = 0$ implies $d(x, T^n x) = 0$.

\begin{remark}

Note that the above definition of fixed point does not affect the usual definition of a convergent sequence, which is already defined only in terms of the values of the pseudometric; a sequence converges to a point $p$ if and only if it converges to any other point $q$ with $d(p, q) = 0$, i.e. convergence of a sequence only matters ``up to distance $0$''. In a complete space this is of course equivalent to the sequence being Cauchy.

\end{remark}

\section{The modified ultraproduct}
\label{section:newultraproduct}

We assume the setting of the ``optionally continuous framework'' (OCF) described in Section \ref{section: setup}. (For clarification about the title of this current section, refer back to Remark \ref{rmk:ref}.)

Let $S$ be an optionally continuous signature. Let $I$ be some index set and $\mathcal{X}_i$ some collection, indexed by $I$, of $S$-structures, each with underlying pseudometric space $(X_i, d_i)$. As usual we assume that our spaces be bounded in diameter by $1$ (if one wants to work with unbounded spaces, one can use sorts to stratify the spaces into bounded spaces of increasing diameter with inclusion maps between them).

Let $\mathcal{F}$ be some ultrafilter on $I$. We take $X = (\prod\limits_{i \in I} X_i ) / \sim_\mathcal{F}$ where we declare that $(x_i) \sim_{\mathcal{F}} (y_i)$ when $\{ i \in I | x_i = y_i\} \in \mathcal{F}$. We denote by $(x_i)_{\mathcal{F}}$ the equivalence class in $X$ represented by $(x_i)$. We still define the pseudometric $d$ on $X$ to be the same as in the usual continuous logic, i.e. $d( (x_i)_{\mathcal{F}}, (y_i)_{\mathcal{F}} ) = \lim\limits_{i, \mathcal{F}} d_i (x_i, y_i)$.

The rest follows naturally: for (possibly discontinuous) functions $f_i: X_i \rightarrow Y_i$, it is clear what $f: X \rightarrow Y$ should be, and that it is well-defined. For predicates $R_i: X_i \rightarrow [0, 1]$, we define $R: X \rightarrow [0,1]$ as $R( (x_i)_\mathcal{F} ) = \lim\limits_{i, \mathcal{F}} R_i (x_i)$.

Now let us assume the setting of the geodesic framework, so that $S$ is now a geodesic signature. The above construction of the ultraproduct in OCF carries over wholesale to this setting; so now we need only to treat the linear structures $L_i$ on $X_i$. If we treat the $L_i$ like (the interpretations of) any other function in our signature, we see that the ultraproduct of the $L_i$ gives a map $X \times X \times ([0,1]^I)/ \sim_\mathcal{F} \, \rightarrow X$. That is, the ultraproduct of the linear structures $L_i$ does \emph{not} specify a linear structure on $X$, and so we see that linear structures must be treated differently under the ultraproduct.

Therefore, given a family $L_i$ of linear structures corresponding to a family of $S$-structures $\mathcal{X}_i$, we do \emph{not} define (the interpretation of) $L$ to be the function $\hat{L}: X \times X \times ([0,1]^I)/ \sim_\mathcal{F} \, \rightarrow X$ that arises as the ultraproduct of the linear structures $L_i$ considered as functions, but rather its restriction across the natural embedding $i_\mathcal{F}: [0,1] \hookrightarrow ([0,1]^I) / \sim_\mathcal{F}$. That is, $L = \hat{L} \circ (1_X, 1_X, i_\mathcal{F}): X \times X \times [0,1] \rightarrow X$. This way, an ultraproduct of spaces equipped with linear structures itself has a linear structure.

It is easy to see that, defined in this way, for each $t \in [0,1]$ the $t$-value $L_t$ of the ultraproduct linear structure $L$ is exactly the ultraproduct of the $t$-values $L^i_t$ treated as binary functions.

Theorem \ref{thm: los} - and therefore also Theorem \ref{thm: avigadiovino} - is still valid in this setting (with the same proof). It is then immediate that the ultraproduct defined in this way is of hyperbolic type if all of its factors are.

\section{Fixed point results for functions satisfying Condition (C) and Condition (E)}
\label{section: analysis}

The definitions and proofs in this section can be found in \cite{GLS}, \cite{suzuki} in the context of Banach spaces, but we will present them here in the more general context of spaces of hyperbolic type, where in many cases no alteration is required, and in some cases only slight adjustments to definitions/proofs are necessary.

Going forward, unless otherwise stated, we denote by $(X,d,L)$ a pseudometric space of hyperbolic type, by $C \subset X$ a (nonempty) subset of $X$, and by $T: C \rightarrow X$ a function from $C$ into $X$ with a priori no special properties (such as continuity).

\newcommand{\condEmu}{condition (E$_\mu$)}
\newcommand{\condCla}{condition (C$_\lambda$)}
\newcommand{\condClap}{condition (C$_{\lambda^\prime}$)}

\newcommand{\CondEmu}{Condition (E$_\mu$)}
\newcommand{\CondCla}{Condition (C$_\lambda$)}
\newcommand{\CondClap}{Condition (C$_{\lambda^\prime}$)}

\newcommand{\CondDla}{Condition (D$_\lambda$)}
\newcommand{\condDla}{condition (D$_\lambda$)}
\newcommand{\condDlap}{condition (D$_{\lambda^\prime}$)}

\begin{definition}
\label{def:afps}

A sequence $\{x_n\}$ of points of $C$ is said to be an \emph{almost fixed point sequence} (or \emph{a.f.p.s.}, for short) for $T$ when $\{x_n\}$ satisfies $d(x_n, Tx_n) \rightarrow 0$.

\end{definition}

\begin{definition}
(\cite{GLS})
\label{def:condemu}

Given $\mu \geq 1$, we say that $T$ \emph{satisfies \condEmu} when $\forall x,y \in C$ we have that $d(x, Ty) \leq \mu \, d(x, Tx) + d(x, y)$.

\end{definition}

We have the following obvious consequence of Definition \ref{def:condemu}:

\begin{proposition}
\textup{(\cite{GLS})}
\label{prop:quasine}

If $T: C \rightarrow X$ satisfies {\condEmu}, and if $x_0 \in C$ is a fixed point of $T$, then for every $x \in C$ we have that $d(x_0, Tx) \leq d(x_0, x)$.

\end{proposition}

The importance of {\condEmu} is that, in the presence of compactness, it guarantees an equivalence between having a fixed point and having an a.f.p.s.:

\begin{theorem}
\label{thm:compactcondemu}
\textup{(\cite{GLS})}

If $C$ is compact and $T: C \rightarrow X$ satisfies {\condEmu}, then $T$ has a fixed point if and only if $T$ admits an a.f.p.s.

\end{theorem}

\begin{proof}
Given an a.f.p.s. $\{x_n\}$, pick a subsequence $\{x_{n_k}\}$ converging to some $x \in C$. We have:
\begin{align*}
\forall k, \; d(x_{n_k}, Tx) &\leq \mu \, d(x_{n_k}, Tx_{n_k}) + d(x_{n_k}, x) \text{, and}\\
\forall k, \; d(x, Tx) &\leq d(x_{n_k}, Tx) + d(x_{n_k}, x)
\end{align*}
which together imply that $d(x, Tx) = 0$.
\end{proof}

It turns out that if $C$ and $T: C \rightarrow X$ are nice in another sense (but $T$ still possibly discontinuous), we can actually guarantee the existence of an a.f.p.s. for $T$:

\begin{definition}
\label{def:condcla}
\cite{suzuki}

Given $\lambda \in [0,1)$, we say that $T: C \rightarrow X$ \emph{satisfies \condCla} when $\forall x, y \in C$, we have that $\lambda \, d(x, Tx) \leq d(x, y)$ implies $d(Tx, Ty) \leq d(x, y)$.

\end{definition}

From the above definition, we see that nonexpansive mappings $T$ are exactly the ones which satisfy {\condCla} for $\lambda = 0$. Also, note that if $\lambda \leq \lambda^\prime$, {\condCla} implies {\condClap}. Therefore, in what follows, we will assume without loss of generality that $\lambda > 0$.

\begin{theorem}
\label{thm:convexcondcla}
\textup{(\cite{GLS}, \cite{GK})}

Let $C$ be a bounded convex subset of $X$, with $T: C \rightarrow C$ satisfying {\condCla}. Then there exists an a.f.p.s. for $T$, namely:

Let $x_1$ be any point in $C$, and let $x_{n+1} = L(x_n, Tx_n, \lambda)$. Then the sequence $\{x_n\}$ is an a.f.p.s. for $T$.

\end{theorem}

The sequence defined above is called a \emph{Mann iteration} for $T$ (starting at $x_1$). To emphasize the role of $\lambda$, we will call it a \emph{$\lambda$-Mann iteration} for $T$ (starting at $x_1$).

The key point to proving this is the following useful lemma, which was originally proven by \cite{GK} for metric spaces of hyperbolic type and then applied to the case of Banach spaces in \cite{suzuki} - and which we now observe actually applies to the more general case of pseudometric spaces of hyperbolic type:

\begin{lemma}
\label{lem:kirkgoebel}
\textup{(\cite{GK}, \cite{suzuki})}

Let $\{x_n\}$ and $\{y_n\}$ be bounded sequences in a pseudometric space $X$ of hyperbolic type, and let $\lambda \in (0,1)$, such that $x_{n+1} = L(x_n, y_n, \lambda)$ and $d(y_{n+1}, y_n) \leq d(x_{n+1}, x_n)$ for all $n$. Then $\lim\limits_{n \rightarrow \infty} d(x_n, y_n) = 0$.

\end{lemma}

The original proof of Lemma \ref{lem:kirkgoebel} applied to metric spaces of hyperbolic type, but the unmodified proof also applies to pseudometric spaces of hyperbolic type. We give the proof, copied essentially verbatim from \cite{GK}, in the Appendix so that the reader may verify this assertion for themselves.

The point is that once we are given the linear structure $L$ on our pseudometric space $X$ which satisfies the hyperbolicity condition, the proof, which is a lengthy string of inequalities, follows entirely mechanically. This is not to say that the proof does not make use of clever manipulations - only that, once the value $\lambda$ and the sequences $\{x_n\}$ and $\{y_n\}$ are specified as in the statement of the Lemma, the proof depends purely on algebraic manipulation of inequalities involving those objects that result from $X$ being a pseudometric space of hyperbolic type, and not, say, any argument that requires points at distance $0$ to be the same point.

\begin{proof}[Proof of Theorem \ref{thm:convexcondcla}]
If we can show that $d(Tx_n, Tx_{n+1}) \leq d(x_n, x_{n+1})$ for each $n$, then the rest is immediate from Lemma \ref{lem:kirkgoebel} by letting $\{y_n\} = \{Tx_n\}$.

Let $n \geq 1$. By construction we have $\lambda \, d(x_n, Tx_n) = d(x_n, x_{n+1})$, so by {\condCla} we have that $d(Tx_n, Tx_{n+1}) \leq d(x_n, x_{n+1})$.
\end{proof}

We have the following fixed point result as a corollary to Theorem \ref{thm:compactcondemu} and Theorem \ref{thm:convexcondcla}:

\begin{corollary}
\label{cor: fixedpt}
\textup{(\cite{GLS})}

If $C$ is a compact, convex subset of $X$, and $T: C \rightarrow C$ satisfies {\condCla} for some $\lambda \in (0,1)$ and {\condEmu} for some $\mu \geq 1$, then $T$ has a fixed point.
\end{corollary}

We have so far looked at properties of maps $T$ satisfying {\condCla} and {\condEmu}. To summarize, Theorem \ref{thm:convexcondcla} shows that {\condCla} along with certain conditions on the domain/codomain of the map $T$ guarantees a sequence which is nice in some asymptotic sense (i.e. is an a.f.p.s.), and then Theorem \ref{thm:compactcondemu} along with compactness of the domain guarantees a fixed point of $T$ to which a subsequence of this a.f.p.s. converges.

Considering that Theorem \ref{thm:convexcondcla} obtains this a.f.p.s. as a $\lambda$-Mann iteration for $T$ for some $\lambda \in (0,1)$, we see that in fact that the entire sequence must converge:

\begin{proposition}
\label{prop:wholeseqconv}

Let $C \subset X$ and $T: C \rightarrow C$ fulfill the conditions of Theorem \ref{thm:compactcondemu} and Theorem \ref{thm:convexcondcla} (with some value $\lambda \in (0,1)$).

Let $\{x_n\}$ be a $\lambda$-Mann iteration for $T$, as given in Theorem \ref{thm:convexcondcla}. Then $\{x_n\}$ converges to the fixed point $x$ guaranteed by Theorem \ref{thm:compactcondemu}.

\end{proposition}

\begin{proof}

We have that $x_{n+1} = L(x_n, Tx_n, \lambda)$. Then by hyperbolicity we have $d(x, x_{n+1}) \leq (1-\lambda) \, d(x, x_n) + \lambda \, d(x, Tx_n)$. By Proposition \ref{prop:quasine} we have that $d(x, Tx_n) \leq d(x, x_n)$ so that $d(x, x_{n+1}) \leq d(x, x_n)$. Since by Theorem \ref{thm:compactcondemu} a subsequence of $\{x_n\}$ converges to $x$, we must have that $\{x_n\}$ itself must converge to $x$.

\end{proof}

We briefly look at a related property of maps $T: C \rightarrow C$ on $C \subset X$, which will serve to illuminate further discussion of {\condCla}:

\begin{definition}
\label{def:dne}
(\cite{dne})

For $C \subset X$, a map $T: C \rightarrow C$ is \emph{directionally nonexpansive} if, for all $\lambda \in [0,1]$ and all $x \in C$, we have that $d(Tx, TL(x,Tx,\lambda)) \leq \lambda \, d(x, Tx)$.

\end{definition}

The moral content of the above definition is that a directionally nonexpansive map $T$ is one that is nonexpansive on the line segment between $x$ and $Tx$.

Furthermore in \cite{dne} Kirk cites \cite{BP} in defining \emph{asymptotic regularity} of $f: C \rightarrow C$ as the condition that for all $x \in C$, $\lim\limits_{n \rightarrow \infty} d(f^n(x), f^{n+1}(x)) = 0$. With this, he proves the following theorem which bears striking resemblance to Theorem \ref{thm:convexcondcla}:

\begin{theorem}
\label{thm:dne}
\textup{(\cite{dne})}

Let $C$ be a bounded convex subset of $X$, and let $T: C \rightarrow C$ be directionally nonexpansive. Fix $\lambda \in (0,1)$, and define $f_T: C \rightarrow C$ by $f_T(x) = L(x, Tx, \lambda)$. Then $f_T$ is asymptotically regular, and this convergence is uniform with respect to the choice of $x$ and $T$.

\end{theorem}

Note that this theorem does not give uniformity with respect to $C$ (however, such uniformity, along with even stronger results about the rate of convergence, is obtained in \cite{kohlenbach3}). Given $x_1 \in C$ and $f_T$ as above, the sequence $\{f_T^n(x_1)\}$ is precisely the $\lambda$-Mann iteration for $T$ starting at $x_1$. Furthermore, asymptotic regularity of $f_T$ is exactly equivalent to the $\lambda$-Mann iteration $\{x_n\} = \{f_T^n(x_1)\}$ being an a.f.p.s. for every starting point $x_1$, since $d(x_n, x_{n+1}) = \lambda d(x_n, Tx_n)$.

\begin{remark}
\label{rmk:clatodne}

The connection just described actually runs deeper. In the proof of Theorem \ref{thm:convexcondcla}, we take a $\lambda$-Mann iteration $\{x_n\}$ and use {\condCla} to conclude, since $\lambda \, d(x_n, Tx_n) = d(x_n, x_{n+1})$, that $d(Tx_n, Tx_{n+1}) \leq d(x_n, x_{n+1})$; then we simply apply Lemma \ref{lem:kirkgoebel} to get that $\{x_n\}$ is an a.f.p.s.

Since by construction we have that $x_{n+1} = L(x_n, Tx_n, \lambda)$, it suffices to forget about {\condCla} and simply require that $d(Tx, TL(x,Tx, \lambda)) \leq d(x, L(x, Tx, \lambda)) = \lambda \, d(x, Tx)$ for all $x \in C$ - call this \emph{\condDla} - to ensure that the proof of Theorem \ref{thm:convexcondcla} nevertheless goes through.

Although {\condDla} simply assumes the conclusion of {\condCla} in the case of e.g. $\lambda$-Mann iterations, {\condDla} is actually \emph{weaker} than {\condCla}. Indeed, for any $x \in C$ we always have that $\lambda \, d(x, Tx) = d(x, L(x, Tx, \lambda))$ so that having {\condCla} would imply that {\condDla} holds.

In light of Definition \ref{def:dne}, we see that {\condDla} is also a weak form of directional nonexpansiveness. Indeed, $T: C \rightarrow C$ is directionally nonexpansive precisely when it satisfies {\condDla} for every $\lambda \in [0,1]$.

\end{remark}

We formalize this discussion as follows:

\begin{definition}
\label{def:conddla}

Given $\lambda \in [0,1]$, we say that $T: C \rightarrow X$ satisfies \emph{\condDla} when $\forall x \in C$, we have that $d(Tx, TL(x,Tx, \lambda)) \leq \lambda \, d(x,Tx)$.

\end{definition}

\begin{proposition}
\label{prop:dlaisweak}

\qquad

\begin{enumerate}[label=(\alph*), ref=(\alph*)]

	\item Given $\lambda \in [0,1)$, if $T: C \rightarrow X$ satisfies {\condCla} then it satisfies {\condDla}. Furthermore, the conclusion of Theorem \ref{thm:convexcondcla} remains true if we require that $T: C \rightarrow C$ satisfy {\condDla} instead of {\condCla} (with the other conditions unchanged).

	\item $T: C \rightarrow C$ is directionally nonexpansive if and only if it satisfies {\condDla} for every $\lambda \in [0,1]$.

\end{enumerate}

\end{proposition}

We note that any other result that we mention that refers to the conclusion of Theorem \ref{thm:convexcondcla} also remains true if we replace {\condCla} with {\condDla}.

One point of caution, however, is that for $\lambda < \lambda^\prime$ we do not necessarily have that {\condDla} implies {\condDlap}. We will nevertheless restrict ourselves to the cases where $\lambda \in (0,1)$, since those are the cases of interest, i.e. the ones to which Theorem \ref{thm:convexcondcla} applies.

\section{Closure under ultraproducts}
\label{section:synthesis}

Thus far, we have reformulated the definitions and results of \cite{GLS}, \cite{GK}, \cite{suzuki} (and \cite{dne} to a lesser extent), which were given in terms of metric/Banach spaces, in terms of pseudometric spaces of hyperbolic type.

We must now formalize all of this in the geodesic framework, which will allow us to use (the geodesic analogue of) Theorem \ref{thm: avigadiovino}, and thus obtain the existence of a uniform bound on the rate of metastable convergence in the results of e.g. Theorem \ref{thm:compactcondemu} and Theorem \ref{thm:convexcondcla}.

Since the argument of Theorem \ref{thm: avigadiovino} crucially requires passing to the ultraproduct, our approach to formalizing the objects and properties discussed in Section \ref{section: analysis} in the geodesic framework will follow the guiding principle that said properties should be preserved under taking ultraproducts.

As with \cite{AI}, whenever we speak of ultrafilters/ultraproducts henceforth, we will assume that the set over which we are taking the ultrafilter is $\mathbb{N}$, and that the ultrafilter is nonprincipal.

We have already seen, via the expression (\ref{hyperbolictypeeqn}), how to formalize the property of a space being of hyperbolic type.

We must currently address two issues which will turn out to have the same solution. The first is that the functions $T: C \rightarrow C$ that we are interested in are only partially defined on $X$ (i.e. they are not functions $T: X \rightarrow X$), and so without further modification cannot be considered honest interpretations of function symbols. The second issue is that, if we are to have a class, closed under taking ultraproducts, of (structures on) spaces $X$ specifying a special subset $C \subset X$ of each space, each of which is required to be convex, then we must formulate some notion of convexity that the subsets $C$ must obey in a manner which is somehow uniform across the members of the class.

The solution is simply to note that the results we are interested in (e.g. Theorem \ref{thm:compactcondemu} and Theorem \ref{thm:convexcondcla}) and their proofs concern themselves only with the features of the subset $C \subset X$. Since we will require $C$ to be convex anyway (so that the linear structure on $X$ restricts to give a linear structure on $C$), we can simply regard $C$ as the entire space. Therefore, in what follows, when we refer to structures $\mathcal{X}$ and the spaces $X$ associated with them, it should be understood that we intend them to play the role of the \emph{subsets} $C \subset X$ from the results of Section \ref{section: analysis}. In this way we get convexity automatically from simply having a linear structure.

Now we would like our ultraproduct to inherit properties such as compactness (which Theorem \ref{thm:compactcondemu} requires) from its factors. More precisely, if we have a family $X_i$ of compact spaces then we would like the ultraproduct $X$ to inherit those properties. This is unfortunately not the case in general. However, in the case of pseudometric spaces, compactness is equivalent to being complete and totally bounded.

We have already started out assuming that our ``base'' spaces will be complete. That their ultraproducts are again complete is simple: ultraproducts are $\omega_1$-saturated, which among other things guarantees Cauchy completeness \cite{HI}.

For total boundedness, we borrow Kohlenbach's idea \cite{kohlenbach1} of specifying a modulus of total boundedness, which is a way to ensure that a family of structures with that modulus is totally bounded in some uniform way. We also give an alternative notion of total boundedness which is equivalent (as Proposotion \ref{prop:tbequiv} will show) yet easier to work with.

\begin{definition}
\label{def:tb}

Let $(X,d)$ be a pseudometric space.

\begin{enumerate}[label=(\alph*), ref=(\alph*)]

	\item We say that $X$ is \emph{totally bounded} when, for every $K \in \mathbb{N}$, there is some $\alpha(K) \in \mathbb{N}$ such that there exist points $x_0, \dots, x_{\alpha(K)}$ such that for all $x \in X$, we have $\min (d(x, x_0), \dots, d(x, x_{\alpha(K)})) < \frac{1}{K+1}$.

This function $\alpha: \mathbb{N} \rightarrow \mathbb{N}$ is called a \emph{modulus of total boundedness} for $X$.

	\item We say that $X$ is \emph{approximately totally bounded} when, for every $k \in \mathbb{N}$, there is some $\beta(k) \in \mathbb{N}$ such that the following holds:
\begin{equation}
\inf\limits_{x_0} \, \cdots \inf\limits_{x_{\beta(k)}} \, \sup\limits_x \, (\min (d(x, x_0), \dots, d(x, x_{\beta(k)}))) \leq \frac{1}{k+1} \label{approxtbeqn}
\end{equation}

This function $\beta: \mathbb{N} \rightarrow \mathbb{N}$ is called a \emph{modulus of approximate total boundedness} for $X$. \label{approxtb}
\end{enumerate}

\end{definition}

Note that the condition given by (\ref{approxtbeqn}) in Definition \ref{def:tb} \ref{approxtb} can be restated as follows:
\begin{equation}
\inf\limits_{x_0} \, \cdots \inf\limits_{x_{\beta(k)}} \, \sup\limits_x \, (\min (d(x, x_0), \dots, d(x, x_{\beta(k)})) \dotminus \frac{1}{k+1}) = 0. \tag{\theequation${}^\prime$}
\end{equation}

So a pseudometric space $X$ is totally bounded if and only if it has some modulus of total boundedness, and approximately totally bounded if and only if it has some modulus of approximate total boundedness. We now show that these two conditions are actually equivalent:

\begin{proposition}
\label{prop:tbequiv}

Let $(X,d)$ be a pseudometric space. The following are equivalent:

\begin{enumerate}[label=(\alph*), ref=(\alph*)]

\item $X$ is totally bounded.

\item $X$ is approximately totally bounded.

\end{enumerate}

\end{proposition}

\begin{proof}

It is clear that $X$ being totally bounded implies that $X$ is approximately totally bounded: if $\alpha$ is a modulus of total boundedness for $X$, $\beta = \alpha$ is a modulus of approximate total boundedness for $X$.

Conversely, let $\beta: \mathbb{N} \rightarrow \mathbb{N}$ be a modulus of approximate total boundedness. We need to produce a function $\alpha: \mathbb{N} \rightarrow \mathbb{N}$ such that for each $K \in \mathbb{N}$, there exist finitely many points $x_1, \dots, x_{\alpha(K)}$ such that for each $x \in X$, we have that $\min (d(x, x_0), \dots, d(x, x_{\alpha(K)})) < \frac{1}{K+1}$. So given $K \in \mathbb{N}$, choose $k \in \mathbb{N}$ to be such that $\frac{2}{k+1} < \frac{1}{K+1}$. Then by assumption there exist points $x_0, \dots, x_{\beta(k)}$ such that for each $x$, we have that $\min (d(x, x_0), \dots, d(x, x_{\beta(k)})) < \frac{2}{k+1} < \frac{1}{K+1}$. Then $\alpha: \mathbb{N} \rightarrow \mathbb{N}$ defined by this assignment $K \mapsto k \mapsto \beta(k)$ is a modulus of total boundedness for $X$.

\end{proof}

The advantage of working with approximate total boundedness is that, as (\ref{approxtbeqn}${}^\prime$) shows, the notion of approximate total boundedness is easily formalized in our setup. In fact, since it neither requires a linear structure nor refers to any discontinuous functions, it is also formalizable in the usual continuous logic - but here we will restrict our discussions to the geodesic framework, which is the one we need to use for our applications.

\begin{definition}
\label{def:logictb}

Let $S$ be a geodesic signature, and $\mathcal{X}$ an $S$-structure.

We say that $\mathcal{X}$ is \emph{totally bounded} when there is a function $\beta: \mathbb{N} \rightarrow \mathbb{N}$ such that for all $k \in \mathbb{N}$, $\mathcal{X}$ satisfies the $S$-condition
\begin{equation*}
\inf\limits_{x_0} \, \cdots \inf\limits_{x_{\beta(k)}} \, \sup\limits_x \, (\min (d(x, x_0), \dots, d(x, x_{\beta(k)})) \dotminus \frac{1}{k+1}) = 0.
\end{equation*}

We call this $\beta$ a \emph{modulus of total boundedness} for $\mathcal{X}$.

\end{definition}

By Proposition \ref{prop:tbequiv}, $\mathcal{X}$ is totally bounded in the above sense if and only if the underlying space $X$ is totally bounded in the sense of Definition \ref{def:tb}.

It is clear from Definition \ref{def:logictb} that if $\beta: \mathbb{N} \rightarrow \mathbb{N}$ is a modulus of total boundedness for a family $\mathcal{X}_i$ of $S$-structures, then $\beta$ is a modulus of total boundedness for the ultraproduct $\mathcal{X}$ of the $\mathcal{X}_i$.

\begin{remark}

Note that even in the absence of total boundedness, any family of $S$-structures for a given signature $S$ automatically shares a bound on the diameters of their underlying spaces, by Definition \ref{def:gl}.

\end{remark}

So far we have seen how to incorporate notions of hyperbolic type, convexity, and compactness into our framework. It remains to express {\condEmu} and {\condDla} as well. It is here that the importance of choosing {\condDla} over {\condCla} becomes clear; while Proposition \ref{prop:dlaisweak} justifies why we \emph{can} do so, the reason why we \emph{want} to is that {\condDla} is much simpler to formalize, because it does not contain any implications (refer to the discussion occurring after Definition \ref{def:convex} for why implications are problematic in our logic).

\begin{definition}
\label{def:logicconds}

Let $S$ be a geodesic signature with a unary function symbol $T$, and let $\mathcal{X}$ be an $S$-structure.

\begin{enumerate}[label=(\alph*), ref=(\alph*)]

	\item Let $\mu \geq 1$. We say that $\mathcal{X}$ \emph{satisfies {\condEmu}} when $\mathcal{X}$ satisfies the $S$-condition
		\begin{equation*}
		\sup\limits_x \, \sup\limits_y \, \left ( (d(x, Ty) \dotminus \mu \, d(x, Tx)) \dotminus d(x,y) \right ) = 0.
		\end{equation*} \label{logiccondemu}

	\item Let $\lambda \in (0,1)$. We say that $\mathcal{X}$ \emph{satisfies {\condDla}} when $\mathcal{X}$ satisfies the $S$-condition
		\begin{equation*}
		\sup\limits_x \, (d(Tx, TL_{\lambda}(x, Tx)) \dotminus \lambda \, d(x, Tx)) = 0.
		\end{equation*} \label{logicconddla}

\end{enumerate}

\end{definition}

Letting $X$ be the underlying space of $\mathcal{X}$ and $T: X \rightarrow X$ the interpretation of the function symbol $T$, it is straightforward to see that $\mathcal{X}$ satisfies Definition \ref{def:logicconds} \ref{logiccondemu} if and only if $T:X \rightarrow X$ satisfies {\condEmu} in the sense of Definition \ref{def:condemu}, since
\begin{align*}
&\forall x,y \in X, \, d(x, Ty) \leq \mu \, d(x, Tx) + d(x, y)\\
\Longleftrightarrow \quad &\sup\limits_x \, \sup\limits_y \, \left ( (d(x, Ty) \dotminus \mu \, d(x, Tx)) \dotminus d(x,y) \right ) = 0.
\end{align*}

Similarly, we see that $\mathcal{X}$ satisfies Definition \ref{def:logicconds} \ref{logicconddla} if and only if $T: X \rightarrow X$ satisfies {\condDla} in the sense of Definition \ref{def:conddla}, since
\begin{align*}
&\forall x \in X, \, d(Tx, TL(x,Tx, \lambda)) \leq \lambda \, d(x,Tx)\\
\Longleftrightarrow \quad &\sup\limits_x \, (d(Tx, TL(x, Tx, \lambda)) \dotminus \lambda \, d(x, Tx)) = 0.
\end{align*}

From this we see that an ultraproduct of structures satisfying {\condDla} (resp. {\condEmu}) itself satisfies {\condDla} (resp. {\condEmu}).

We are now ready to apply the Avigad-Iovino approach to Theorem \ref{thm:convexcondcla}.

\begin{theorem}
\label{thm:main1}

Let $S$ be a geodesic signature with a unary function symbol $T$ and a constant symbol $x_1$.

Let $\lambda \in (0,1)$ be given, and let $\mathcal{C}$ be the class of $S$-structures of hyperbolic type satisfying {\condDla}.

For each $\mathcal{X} \in \mathcal{C}$, let $\{x_n\}$ be the sequence defined by $x_{n+1} = L_{\lambda}(x_n, Tx_n)$. Then we have the following:

\begin{enumerate}[label=(\alph*), ref=(\alph*)]

\item Letting $d_n = d(x_n, x_{n+1})$, we have that $\lim\limits_{n \rightarrow \infty} d_n = 0$. Equivalently, $\{x_n\}$ is an a.f.p.s. for $T$. \label{main1afps}

\item Given any function $F: \mathbb{N} \rightarrow \mathbb{N}$, there is a bound on the rate of metastability of the above convergence with respect to $F$, which is uniform in $\mathcal{X} \in \mathcal{C}$. \label{main1uniform}

\end{enumerate}

\end{theorem}

\begin{remark}
\label{rmk:explainmain1}

Recall from Definition \ref{def: metastability} that, given $F: \mathbb{N} \rightarrow \mathbb{N}$, a \emph{bound on the rate of metastability} for a sequence $\{d_n\}$ is a function $b_F: \mathbb{R}_{>0} \rightarrow \mathbb{N}$ such that for each $\epsilon > 0$ there exists some $n \leq b_F(\epsilon)$ such that for all $i, j \in [n, F(n)]$, we have that $d(d_i, d_j) < \epsilon$. (In the specific case of Theorem \ref{thm:main1}, $d(d_i, d_j) = |d_i - d_j|$.)

\end{remark}

Theorem \ref{thm:main1} is a simultaneous generalization of Theorem \ref{thm:convexcondcla} and Theorem \ref{thm:dne}, since {\condDla} is a weaker condition than both {\condCla} (used in Theorem \ref{thm:convexcondcla}) and directional nonexpansiveness (used in Theorem \ref{thm:dne}).

Furthermore, since the data of each structure $\mathcal{X}$ includes not only the space $X$ but also the function $T: X \rightarrow X$ as well as the choice of starting point $x_1 \in X$, Theorem \ref{thm:main1} \ref{main1uniform} guarantees a bound on the ``metastable asymptotic regularity'' of the $\lambda$-Mann iterations that is uniform in $X$, functions $T: X \rightarrow X$, and starting points $x_1 \in X$.

Proving Theorem \ref{thm:main1} will involve the following lemma which is a variant of Theorem \ref{thm: avigadiovino}:

\begin{lemma}
\label{lem: cheeky}

Let $S$ be a geodesic signature, and let $\{t_n\}$ be a sequence of $S$-terms.

Let $\mathcal{C}$ a class of $S$-structures. For each $\mathcal{X} \in \mathcal{C}$, let $\{x_n\}$ denote the interpretation in $\mathcal{X}$ of the sequence $\{t_n\}$, and let $d_n = d(x_n, x_{n+1})$.

Finally, let $\mathcal{F}$ be an ultrafilter. Then the following are equivalent:

\begin{enumerate}[label=(\alph*), ref=(\alph*)]

\item For every $\epsilon > 0$ and every $F: \mathbb{N} \rightarrow \mathbb{N}$, there is some $b \geq 1$ such that the following holds: for every $\mathcal{X}$ in $\mathcal{C}$, there is an $n \leq b$ such that $d_i < \epsilon$ for every $i \in [n, F(n)]$. \label{cheeky1}

\item For any sequence $\mathcal{X}_k$ of elements of $\mathcal{C}$, let $\mathcal{X}$ be their $\mathcal{F}$-ultraproduct. Then for every $\epsilon > 0$ and every $F: \mathbb{N} \rightarrow \mathbb{N}$, there is an $n \in \mathbb{N}$ such that $d_i < \epsilon$ for every $i \in [n, F(n)]$. \label{cheeky2}

\end{enumerate}

\end{lemma}

\begin{proof}

The proof is essentially the same as for Theorem \ref{thm: avigadiovino}.

\ref{cheeky1} $\Rightarrow$ \ref{cheeky2}: For any fixed $\frac{1}{2} \epsilon > 0$ and any fixed $F: \mathbb{N} \rightarrow \mathbb{N}$, there is a $b \geq 1$ such that every member of $\mathcal{C}$ satisfies the condition

$\min\limits_{n \leq b} \left ( \max \limits_{i \in [n, F(n)]} ( d_i \dotminus \frac{1}{2} \epsilon ) \right ) = 0$.

Thus any ultraproduct of members of $\mathcal{C}$ must again be a model of this sentence.

\ref{cheeky2} $\Rightarrow$ \ref{cheeky1}: Assume that \ref{cheeky1} fails. That is, for some $\epsilon > 0$ and some $F: \mathbb{N} \rightarrow \mathbb{N}$, for each $k \in \mathbb{N}$ there is a member $\mathcal{X}_k$ of $\mathcal{C}$ such that for every $n \leq k$ and for some $i \in [n, F(n)]$, we have $d^k_i \geq \epsilon$. Let $\mathcal{X}$ be the $\mathcal{F}$-ultraproduct of the sequence $\mathcal{X}_k$ thus obtained.

Given any $n$, since there are cofinitely many $k \geq n$, it is also true for cofinitely many $k$ that there is some $i \in [n, F(n)]$ with $d^k_i \geq \epsilon$. It follows that there is some specific $i \in [n, F(n)]$ such that $d^k_i \geq \epsilon$ for $\mathcal{F}$-many $k$, so that $d_i = \lim\limits_{k, \mathcal{F}} d^k_i \geq \epsilon$ for that $i$. Since $n$ was arbitrary, we see that \ref{cheeky2} fails.

\end{proof}

\begin{proof}[Proof of Theorem \ref{thm:main1}]

\qquad

\ref{main1afps}: For each $\mathcal{X} \in \mathcal{C}$, $\{x_n\}$ is the $\lambda$-Mann iteration for $T$ starting at $x_1$. By having specified a geodesic signature $S$ we automatically have that the underlying space $X$ is bounded and convex with respect to the linear structure. Since $\mathcal{X}$ is of hyperbolic type and $T: X \rightarrow X$ satisfies {\condDla}, we can use Proposition \ref{prop:dlaisweak} to apply Theorem \ref{thm:convexcondcla} and conclude that $\{x_n\}$ is an a.f.p.s. for $T$. And since $d_n = \lambda \, d(x_n, Tx_n)$, we equivalently have that $\lim\limits_{n \rightarrow \infty} d_n = 0$.

\ref{main1uniform}: For each $\mathcal{X} \in \mathcal{C}$, the sequence $\{x_n\}$ is the interpretation in $\mathcal{X}$ of the sequence of $S$-terms $\{t_n\}$ where $t_1 = x_1$ and $t_{n+1} = L_\lambda (t_n, Tt_n)$.

Furthermore, given an ultrafilter $\mathcal{F}$, for any sequence $\mathcal{X}_k$ of elements of $\mathcal{C}$, their $\mathcal{F}$-ultraproduct $\mathcal{X}$ is again an $S$-structure (so bounded and convex) of hyperbolic type satisfying {\condDla}, so that $\lim\limits_{n \rightarrow \infty} d_n = 0$. By Proposition \ref{prop: cauchy=meta}, we see that part \ref{cheeky2} of Lemma \ref{lem: cheeky} is satisfied, so that we have part \ref{cheeky1} of that lemma as well, which gives us Theorem \ref{thm:main1} \ref{main1uniform}.

\end{proof}

Now that we have obtained a uniform version of Theorem \ref{thm:convexcondcla}, we consider the case where we also have compactness (total boundedness) and \condEmu{}:

\begin{theorem}
\label{thm:main2}

Let $S$ be a geodesic signature with a unary function symbol $T$ and a constant symbol $x_1$.

Let $\lambda \in (0,1)$, $\mu \geq 1$, and $\beta: \mathbb{N} \rightarrow \mathbb{N}$ be given.

Let $\mathcal{C}$ be the class of $S$-structures of hyperbolic type which have $\beta$ as a modulus of total boundedness, and which satisfy {\condDla} and {\condEmu}.

Finally, for each $\mathcal{X} \in \mathcal{C}$, let $\{x_n\}$ be the sequence defined by $x_{n+1} = L_\lambda(x_n, Tx_n)$. Then we have the following:

\begin{enumerate}[label=(\alph*),ref=(\alph*)]

	\item $T$ has a fixed point toward which $\{x_n\}$ converges. \label{main2fp}

	\item For each $F: \mathbb{N} \rightarrow \mathbb{N}$, there is a bound on the rate of metastability for the above convergence which is uniform in $\mathcal{X} \in \mathcal{C}$. \label{main2uniform}

\end{enumerate}

\end{theorem}

\begin{proof}

\ref{main2fp}: For each $\mathcal{X} \in \mathcal{C}$, we have that $\{x_n\}$ is the $\lambda$-Mann iteration for $T$ starting at $x_1$. The underlying space $X$ is convex and compact, $\mathcal{X}$ is of hyperbolic type, and $T: X \rightarrow X$ satisfies {\condDla} and {\condEmu}, so $\{x_n\}$ is an a.f.p.s. for $T$, which then converges to a fixed point $x$ by Theorem \ref{thm:compactcondemu} and Proposition \ref{prop:wholeseqconv}.

\ref{main2uniform}: As in the proof of Theorem \ref{thm:main1} \ref{main1uniform}, for each $\mathcal{X}$ the sequence $\{x_n\}$ is the interpretation of the sequence $\{t_n\}$ of $S$-terms where $t_1 = x_1$ and $t_{n+1} = L_\lambda (t_n, Tt_n)$.

All of the relevant conditions - convexity, compactness, hyperbolic type, {\condDla}, and {\condEmu} - are preserved under ultraproducts. Thus given an ultrafilter $\mathcal{F}$ and any sequence $\mathcal{X}_k$ of elements of $\mathcal{C}$, the $\mathcal{F}$-ultraproduct $\mathcal{X}$ of the $\mathcal{X}_k$ is again in $\mathcal{C}$, so that the sequence $\{x_n\}$ associated with $\mathcal{X}$ converges. Thus by Theorem \ref{thm: avigadiovino} (which, as we observed at the end of Section \ref{section:newultraproduct}, is still valid for geodesic logic) we have Theorem \ref{thm:main2} \ref{main2uniform}.

\end{proof}

\begin{remark}
\label{rmk:closing}

As in Theorem \ref{thm:main1}, the bound on the rate of metastability guaranteed by the theorem above is uniform in the spaces $X$, functions $T: X \rightarrow X$, and choices of starting point $x_1 \in X$ for the $\lambda$-Mann iterations.

\end{remark}

\section{Appendix}
\label{appendix}

Here we supply the proof of Lemma \ref{lem:kirkgoebel}, to make it clear that the entire proof is valid, unmodified from \cite{GK}, within the context of pseudometric spaces of hyperbolic type.

\begin{lem:kirkgoebel}
\textup{(\cite{GK}, \cite{suzuki})}

Let $\{x_n\}$ and $\{y_n\}$ be bounded sequences in a pseudometric space $X$ of hyperbolic type, and let $\lambda \in (0,1)$, such that $x_{n+1} = L(x_n, y_n, \lambda)$ and $d(y_{n+1}, y_n) \leq d(x_{n+1}, x_n)$ for all $n$. Then $\lim\limits_{n \rightarrow \infty} d(x_n, y_n) = 0$.

\end{lem:kirkgoebel}

\begin{proof}

The first claim is that, for all $i, n \in \mathbb{N}$:
\begin{equation}
\label{1stkirkgoebeleqn}
(1 + n \lambda) \, d(x_i, y_i) \leq d(x_i, y_{i+n}) + (1 - \lambda)^{-n} (d(x_i, y_i) - d(x_{i+n}, y_{i+n}))
\end{equation}

If $n = 1$, then (\ref{1stkirkgoebeleqn}) simplifies to $(1+ \lambda) \, d(x_i, y_i) \leq d(x_i, y_{i+1}) + \frac{1}{1- \lambda} (d(x_i, y_i) - d(x_{i+1}, y_{i+1}))$, which we can manipulate as follows:
%$d(x_{i+1}, y_{i+1}) \leq d(x_i, y_{i+1})$, which is true by hyperbolicity and our assumptions:
\begin{align*}
(1+ \lambda) \, d(x_i, y_i) &\leq d(x_i, y_{i+1}) + \frac{1}{1- \lambda} (d(x_i, y_i) - d(x_{i+1}, y_{i+1}))\\
\Longleftrightarrow \quad d(x_{i+1}, y_{i+1}) &\leq (1-\lambda) d(x_i, y_{i+1}) + \lambda^2 \, d(x_i, y_i)\\
&= (1-\lambda) d(x_i, y_{i+1}) + \lambda \, d(x_i, x_{i+1})
%\Longleftrightarrow \quad d(x_i, y_i) - \lambda^2 \, d(x_i, y_i) &\leq (1- \lambda) d(x_i, y_{i+1}) + d(x_i, y_i) - d(x_{i+1}, y_{i+1})\\
%d(x_{i+1}, y_{i+1}) &\leq (1-\lambda) d(x_i, y_{i+1}) + \lambda \, d(y_i, y_{i+1})\\
%&\leq (1-\lambda) d(x_i, y_{i+1}) + \lambda \, d(x_i, x_{i+1})\\
%&= (1-\lambda) d(x_i, y_{i+1}) + \lambda^2 d(x_i, y_i)\\
\end{align*}
where we have used the fact that $d(x_i, x_{i+1}) = \lambda \, d(x_i, y_i)$. But we know that $d(x_{i+1}, y_{i+1}) \leq (1-\lambda) d(x_i, y_{i+1}) + \lambda \, d(x_i, x_{i+1})$ by hyperbolicity, so (\ref{1stkirkgoebeleqn}) holds for $n=1$ and all $i \in \mathbb{N}$.

%The rest of the proof (of this first claim!) uses induction on $n$ and is quite long, so we include it in an appendix to this paper, along with the rest of the proof which uses the claim and the conditions of the Lemma to conclude that $\lim\limits_{n \rightarrow \infty} d(x_n, y_n) = 0$.

So let us assume by induction that (\ref{1stkirkgoebeleqn}) is true for some $n$, and all $i$. By replacing $i$ with $i+1$, we get:
\begin{align}
\label{2ndkirkgoebeleqn}
\begin{split}
(1+n \lambda) \, d(x_{i+1}, y_{i+1}) &\leq d(x_{i+1}, y_{i+n+1}) \\
& \qquad + (1-\lambda)^{-n} (d(x_{i+1}, y_{i+1}) - d(x_{i+n+1}, y_{i+n+1}))
\end{split}
\end{align}

while from hyperbolicity and the rest of our assumptions we get:
\begin{align}
\label{3rdkirkgoebeleqn}
\begin{split}
d(x_{i+1}, y_{i+n+1}) &\leq (1 - \lambda) \, d(x_i, y_{i+n+1}) + \lambda \, d(y_i, y_{i+n+1})\\
&\leq (1 - \lambda) \, d(x_i, y_{i+n+1}) + \lambda \sum\limits^{i+n}_{k=i} d(y_k, y_{k+1})\\
&\leq (1 - \lambda) \, d(x_i, y_{i+n+1}) + \lambda \sum\limits^{i+n}_{k=i} d(x_k, x_{k+1})
\end{split}
\end{align}

It is easy to check that our assumptions imply that $d(x_k,y_k) \geq d(x_{k+1},y_{k+1})$ for all $k$. We use this fact and the aforementioned assumptions in the following derivation which combines (\ref{2ndkirkgoebeleqn}) and (\ref{3rdkirkgoebeleqn}):
\begin{align*}
d(x_i, y_{i+n+1}) &\geq (1-\lambda)^{-1} d(x_{i+1}, y_{i+n+1}) - \lambda(1-\lambda)^{-1} \sum\limits^{i+n}_{k=i} d(x_k, x_{k+1})\\
&\geq  (1-\lambda)^{-1} (1+ n\lambda) \, d(x_{i+1}, y_{i+1})\\
&\qquad + (1-\lambda)^{-n-1} \left ( d(x_{i+n+1}, y_{i+n+1}) - d(x_{i+1}, y_{i+1}) \right )\\
&\qquad  - \lambda (1-\lambda)^{-1} \sum\limits^{i+n}_{k=i} d(x_k, x_{k+1})\\
&=  (1-\lambda)^{-1} (1+ n\lambda) \, d(x_{i+1}, y_{i+1})\\
&\qquad + (1-\lambda)^{-n-1} \left ( d(x_{i+n+1}, y_{i+n+1}) - d(x_{i+1}, y_{i+1}) \right )\\
&\qquad - \lambda^2 (1-\lambda)^{-1} \sum\limits^{i+n}_{k=i} d(x_k, y_k)\\
&\geq (1-\lambda)^{-1} (1+ n\lambda) \, d(x_{i+1}, y_{i+1})\\
&\qquad + (1-\lambda)^{-n-1} \left ( d(x_{i+n+1}, y_{i+n+1}) - d(x_{i+1}, y_{i+1}) \right )\\
&\qquad - \lambda^2 (1-\lambda)^{-1} (n+1) \, d(x_i, y_i)\\
&= (1-\lambda)^{-n-1} \left ( d(x_{i+n+1}, y_{i+n+1}) - d(x_i, y_i) \right)\\
&\qquad + (1-\lambda)^{-1} \left ( (1 + n\lambda) - (1-\lambda)^{-n} \right ) \, d(x_{i+1}, y_{i+1})\\
&\qquad + \left ( (1-\lambda)^{-n-1} - \lambda^2 (1-\lambda)^{-1} (n+1) \right) \, d(x_i, y_i)
\end{align*}

From e.g. the expression of each $\frac{1}{1-\lambda}$ as a power series, we have that $(1+n\lambda) \leq (1-\lambda)^{-n}$, so the last inequality above remains true when we replace $d(x_{i+1}, y_{i+1})$ by $d(x_i, y_i)$:
\begin{align*}
d(x_i, y_{i+n+1}) &\geq (1-\lambda)^{-n-1} \left ( d(x_{i+n+1}, y_{i+n+1}) - d(x_i, y_i) \right)\\
&\qquad + (1-\lambda)^{-1} \left ((1+n\lambda) - \lambda^2 (n+1) \right ) \, d(x_i, y_i)\\
&= (1-\lambda)^{-(n+1)} \left ( d(x_{i+n+1}, y_{i+n+1}) - d(x_i, y_i) \right)\\
&\qquad + (1 + (n+1)\lambda) \, d(x_i, y_i)
\end{align*}

which completes the induction.

Having proven (\ref{1stkirkgoebeleqn}), we now show that $\lim\limits_{n \rightarrow \infty} d(x_n, y_n) = 0$.

Assume otherwise, i.e. that there is some $r > 0$ such that $\lim\limits_{n \rightarrow \infty} d(x_n, y_n) = r$. Let $D$ denote a bound for the sequences $\{x_n\}$ and $\{y_n\}$.

We can pick $\epsilon > 0$ such that $\epsilon \exp{\left ( (1-\lambda)^{-1} (r^{-1} D + 1) \right )} < r$.

Choose $i$ so that for all $n \geq 1$, $d(x_i, y_i) - d(x_{i+n}, y_{i+n}) \leq \epsilon$, and choose $N$ so that $\lambda r (N-1) \leq D \leq \lambda r N$. Then we have that $\lambda r N < D + r$ ($\Rightarrow N \lambda < r^{-1}D + 1$).

We also have:

\begin{align*}
(1-\lambda)^{-N} &= (1 + (1-\lambda)^{-1})^{N}\\
&= \exp{ \left( N \log (1+\lambda (1-\lambda)^{-1}) \right) }\\
&\leq \exp{ \left ( N \lambda(1-\lambda)^{-1} \right ) }
\end{align*}

So that we get the following contradiction:

\begin{align*}
D+r \leq (1+ N \lambda) r &\leq (1+N \lambda) d(x_i, y_i)\\
&\leq d(x_i, y_{i+N}) + \epsilon \exp{ \left ( N \lambda(1-\lambda)^{-1} \right ) }\\
&\leq D + \epsilon \exp{\left ( (1-\lambda)^{-1} (r^{-1} D + 1) \right )}\\
&< D + r
\end{align*}
\end{proof}

\end{document}